\documentclass[11pt]{article}

\usepackage{amsmath, amsfonts, amssymb, amsthm,  graphicx, mathtools, enumerate,float}
\usepackage{mathrsfs}
\usepackage{enumitem}
\usepackage{blindtext}

\usepackage[blocks, affil-it]{authblk}

\usepackage[numbers, square]{natbib}
\usepackage[CJKbookmarks=true,
            bookmarksnumbered=true,
			bookmarksopen=true,
			colorlinks=true,
			citecolor=red,
			linkcolor=blue,
			anchorcolor=red,
			urlcolor=blue]{hyperref}
\usepackage[usenames]{color}

\usepackage[letterpaper, left=1.2truein, right=1.2truein, top = 1.2truein, bottom = 1.2truein]{geometry}
\usepackage[ruled, vlined, lined, commentsnumbered]{algorithm2e}

\usepackage{prettyref,soul}

\newtheorem{lemma}{Lemma}[section]
\newtheorem{proposition}{Proposition}[section]
\newtheorem{thm}{Theorem}[section]

\newtheorem{example}{Example}[section]

\newrefformat{eq}{(\ref{#1})}
\newrefformat{chap}{Chapter~\ref{#1}}
\newrefformat{sec}{Section~\ref{#1}}
\newrefformat{algo}{Algorithm~\ref{#1}}
\newrefformat{fig}{Fig.~\ref{#1}}
\newrefformat{tab}{Table~\ref{#1}}
\newrefformat{rmk}{Remark~\ref{#1}}
\newrefformat{clm}{Claim~\ref{#1}}
\newrefformat{def}{Definition~\ref{#1}}
\newrefformat{cor}{Corollary~\ref{#1}}
\newrefformat{lmm}{Lemma~\ref{#1}}
\newrefformat{lemma}{Lemma~\ref{#1}}
\newrefformat{prop}{Proposition~\ref{#1}}
\newrefformat{app}{Appendix~\ref{#1}}
\newrefformat{ex}{Example~\ref{#1}}
\newrefformat{exer}{Exercise~\ref{#1}}
\newrefformat{soln}{Solution~\ref{#1}}
\newrefformat{cond}{Condition~\ref{#1}}

\def\text#1{\mbox{\rm #1}}

\def\X{\mathscr{X}}
\def\Z{\mathcal{Z}}

\DeclarePairedDelimiter{\ceil}{\lceil}{\rceil}

\newcommand{\argmin}{\mathop{\rm argmin}}
\newcommand{\argmax}{\mathop{\rm argmax}}

\newcommand{\indc}[1]{{\mathbf{1}_{\left\{{#1}\right\}}}}
\newcommand{\norm}[1]{\|{#1} \|}

\newcommand{\wh}{\widehat}
\newcommand{\wt}{\widetilde}

\newcommand{\R}{\mathbb{R}}

\newcommand{\floor}[1]{{\left\lfloor {#1} \right \rfloor}}

\title{Convergence Rates of Empirical Bayes Posterior Distributions: A Variational Perspective
}
\author{Fengshuo Zhang}
\author{Chao Gao}
\affil{University of Chicago
}

\begin{document}
\maketitle

\begin{abstract}
We study the convergence rates of empirical Bayes posterior distributions for nonparametric and high-dimensional inference. We show that as long as the hyperparameter set is discrete, the empirical Bayes posterior distribution induced by the maximum marginal likelihood estimator can be regarded as a variational approximation to a hierarchical Bayes posterior distribution. This connection between empirical Bayes and variational Bayes allows us to leverage the recent results in the variational Bayes literature \citep{alquier2020concentration,yang2020alpha,zhang2017convergence}, and directly obtains the convergence rates of empirical Bayes posterior distributions from a variational perspective. For a more general hyperparameter set that is not necessarily discrete, we introduce a new technique called ``prior decomposition" to deal with prior distributions that can be written as convex combinations of probability measures whose supports are low-dimensional subspaces. This leads to generalized versions of the classical ``prior mass and testing" conditions for the convergence rates of empirical Bayes. Our theory is applied to a number of statistical estimation problems including nonparametric density estimation and sparse linear regression.

\smallskip

\end{abstract}



\section{Introduction}

Given a likelihood function $p(X|\theta)$ and a prior $\theta\sim \Pi_{\lambda}$, the posterior distribution can be calculated via the Bayes formula
\begin{equation}
d\Pi_{\lambda}(\theta|X)\propto p(X|\theta)d\Pi_{\lambda}(\theta). \label{eq:basic-lambda-setting}
\end{equation}
In many statistical estimation problems, the prior is usually indexed by a hyperparameter $\lambda$ that controls the regularity of the distribution. To achieve minimax optimal estimation error from a frequentist perspective, the hyperparameter $\lambda$ should be selected according to the underlying structure of the data generating process. One popular selection method is the empirical Bayes principle. That is, find
\begin{equation}
\wh{\lambda} = \argmax_{\lambda\in\Lambda}\log\left[w(\lambda)\int p(X|\theta)d\Pi_{\lambda}(\theta)\right], \label{eq:MMLE-EB}
\end{equation}
and then use the distribution $\Pi_{\wh{\lambda}}(\cdot|X)$ for posterior inference. The formula (\ref{eq:MMLE-EB}) is known as the maximum marginal likelihood estimator (MMLE) when $w(\lambda)=1$ for all $\lambda\in\Lambda$. One can also use a more general weight function $w(\lambda)$ to reflect the prior knowledge of the space $\Lambda$. In this paper, we study frequentist properties of the empirical Bayes posterior distribution $\Pi_{\wh{\lambda}}(\cdot|X)$ by assuming a frequentist data generating process $X\sim P^*$.

Our main approach relies on the recent progress in the theoretical analysis of variational Bayes posterior distributions \citep{alquier2020concentration,yang2020alpha,zhang2017convergence}. Given a variational class $\mathcal{S}$, a set of distributions, the variational approximation to a posterior distribution $\Pi(\cdot|X)$ is defined by
\begin{equation}
\wh{Q}=\argmin_{Q\in\mathcal{S}}D(Q\|\Pi(\cdot|X)), \label{eq:VB}
\end{equation}
where $D(\cdot\|\cdot)$ is the Kullback-Leibler divergence. The data-dependent probability measure $\wh{Q}$ is called variational posterior distribution, and is widely used in the context of machine learning and complex high-dimensional models because of the potential computational intractability of the posterior distribution \citep{blei2017variational}. The recent work \cite{alquier2020concentration,yang2020alpha,zhang2017convergence} formulate conditions on prior, likelihood, and variational class, under which convergence rates of variational posterior distributions can be established in general settings of nonparametric and high-dimensional estimation. In addition, it is observed in \cite{zhang2017convergence} that for a very special class of models, the empirical Bayes posterior distribution can be regarded as a variational posterior distribution for a specific choice of $\mathcal{S}$ in (\ref{eq:VB}). It is further suggested by \cite{zhang2017convergence} that this connection between variational Bayes and empirical Bayes may lead to results on the convergence rates of the empirical Bayes posterior as well.

In this paper, we follow the suggestion of \cite{zhang2017convergence} and formally establish this connection between variational Bayes and empirical Bayes in a more general setting. As a result, the theoretical properties of the variational posterior proved in \cite{alquier2020concentration,yang2020alpha,zhang2017convergence} are easily applied to establish the convergence rates of the empirical Bayes posterior distributions. We show that as long as the hyperparameter set $\Lambda$ is discrete and $\sum_{\lambda\in\Lambda}w(\lambda)<\infty$, there exists a variational class $\mathcal{S}_{\rm EB}$, such that the empirical Bayes posterior $\Pi_{\wh{\lambda}}(\cdot|X)$ can be equivalently written as (\ref{eq:VB}) with $\mathcal{S}=\mathcal{S}_{\rm EB}$. The posterior distribution $\Pi(\cdot|X)$ in (\ref{eq:VB}) is given by $d\Pi(\cdot|X)\propto p(X|\theta)d\Pi(\theta)$, where
$$\Pi=\frac{\sum_{\lambda\in\Lambda}w(\lambda)\Pi_{\lambda}}{\sum_{\lambda\in\Lambda}w(\lambda)}.$$
In other words, the empirical Bayes posterior $\Pi_{\wh{\lambda}}(\cdot|X)$ can be regarded as a variational approximation to a hierarchical Bayes posterior distribution. This connection automatically makes the results on the variational posterior convergence rates directly applicable to a general class of empirical Bayes posterior distributions. Moreover, since the conditions of \cite{alquier2020concentration,yang2020alpha,zhang2017convergence} are formulated in the classical ``prior mass and testing" style \citep{schwartz1965bayes,barron1988exponential,barron1999consistency,ghosal1999posterior,shen2001rates,ghosal2000convergence,ghosal2007convergence}, the obtained theory of empirical Bayes posterior uses the same set of conditions, and thus can be easily verified in many important nonparametric and high-dimensional estimation problems.

In addition to the theory of discrete $\Lambda$, we also propose a method of ``prior decomposition" to analyze the empirical Bayes posterior when the hyperparameter set $\Lambda$ is continuous (uncountable). For a collection of priors $\{\Pi_{\lambda}:\lambda\in\Lambda\}$ that can be decomposed as convex combinations of probability distributions whose supports are low-dimensional subspaces, we extend the ``prior mass and testing" conditions in \cite{schwartz1965bayes,barron1988exponential,barron1999consistency,ghosal1999posterior,shen2001rates,ghosal2000convergence,ghosal2007convergence} to derive sharp convergence rates for the empirical Bayes posterior distributions. The theory is applied to a number of high-dimensional estimation problems including the popular spike-and-slab priors for sparse estimation.

\paragraph{Connection to the literature.}

Convergence rates of empirical Bayes posterior distributions have been investigated in a number of settings in the literature. This includes the selection of hyperparameters of the spike-and-slab prior \citep{johnstone2004needles,castillo2018empirical,castillo2020spike,castillo2018spike} and the horseshoe prior \citep{van2014horseshoe,van2017uncertainty} for sparse sequence model, the scaling of Gaussian processes \citep{szabo2013empirical,szabo2015frequentist,knapik2016bayes,sniekers2015adaptive} for both nonparametric regression and inverse problems, and empirical Bayesian model selection \citep{rousseau2016asymptotic} for sieve priors. Theoretical properties of $\Pi_{\wh{\lambda}}(\cdot|X)$ in general settings have also been established by \cite{petrone2014bayes,donnet2018posterior,rousseau2017asymptotic}. In particular, \cite{petrone2014bayes} studied the asymptotic behavior of $\Pi_{\wh{\lambda}}(\cdot|X)$ for general parametric models, while \cite{donnet2018posterior} provides sufficient conditions for the convergence rates of $\Pi_{\wh{\lambda}}(\cdot|X)$ in nonparametric settings when $\wh{\lambda}$ is known to belong to a set $\Lambda_0$ that has nice properties. Perhaps the most general result for nonparametric models is the work \cite{rousseau2017asymptotic}. Sufficient conditions were formulated by \cite{rousseau2017asymptotic} to prove $\wh{\lambda}\in\Lambda_0$ with high probability with $\Lambda_0=\{\lambda:\epsilon(\lambda)\lesssim \epsilon_0\}$, where $\epsilon(\lambda)$ is understood to be the convergence rate of the posterior distribution $\Pi_{\lambda}(\cdot|X)$, and $\epsilon_0$ is the convergence rate to be established for the empirical Bayes posterior. The result $\wh{\lambda}\in\Lambda_0$, together with a standard analysis for $\sup_{\lambda\in\Lambda_0}\Pi_{\lambda}(\cdot|X)$, leads to the desired convergence rate for $\Pi_{\wh{\lambda}}(\cdot|X)$.

Despite the generality of \cite{rousseau2017asymptotic}, checking the sufficient conditions that lead to $\wh{\lambda}\in\Lambda_0$ are usually quite difficult. For each example of applications, one needs to first construct a very specific measure that is not necessarily a probability distribution, and then establish a testing error condition under this measure. In comparison, the classical ``prior mass and testing" conditions \citep{schwartz1965bayes,barron1988exponential,barron1999consistency,ghosal1999posterior,shen2001rates,ghosal2000convergence,ghosal2007convergence} for posterior convergence rates work with the likelihood directly and the conditions are much more straightforward to check. The conditions in our theory, derived from a variational approximation perspective, are almost identical to the classical ``prior mass and testing" conditions. This leads to some significant simplifications of \cite{rousseau2017asymptotic} when applying the theory to specific examples. Moreover, for the more general continuous hyperparameter set, the proposed method of ``prior decomposition" leads to conditions that can be applied to a number of high-dimensional models with unbounded parameter spaces. For these examples, we believe the theory of \cite{rousseau2017asymptotic}  will lead to unnecessary logarithmic factors for the convergence rates because of the unboundedness of the model parameters.

Last but not least, let us emphasize that this paper only deals with empirical Bayes procedures defined by the MMLE (\ref{eq:MMLE-EB}). Sometimes the terminology is also used for general data-dependent probability measures that summarize the information of likelihood and prior. For this line of research, we refer the readers to \cite{martin2014asymptotically,martin2017empirical,martin2019data,belitser2019empirical,belitser2019general} and references therein.

\paragraph{Paper organization.} The rest of the paper is organized as follows. In Section \ref{sec:VB}, we review the recent theory of convergence rates for variational posterior distributions. Then, we formally establish the connection between empirical Bayes and variational Bayes in Section \ref{sec:EB-VB} and derive the convergence rates of empirical Bayes posterior distributions in the same section. The result is generalized to continuous hyperparameter set and unbounded parameter space in Section \ref{sec:gen}. Finally in Section \ref{sec:pf}, additional proofs of all technical results are presented.

\paragraph{Notation.} For an integer $d$, we use $[d]$ to denote the set $\{1,2,...,d\}$. Given two numbers $a,b\in\mathbb{R}$, we use $a\vee b=\max(a,b)$ and $a\wedge b=\min(a,b)$. We also write $a_+=\max(a,0)$. For two positive sequences $\{a_n\},\{b_n\}$, $a_n\lesssim b_n$ or $a_n=O(b_n)$ means $a_n\leq Cb_n$ for some constant $C>0$ independent of $n$, $a_n=\Omega(b_n)$ means $b_n=O(a_n)$, and $a_n\asymp b_n$ means $a_n\lesssim b_n$ and $b_n\lesssim a_n$. We also write $a_n=o(b_n)$ when $\limsup_n\frac{a_n}{b_n}=0$. For a set $S$, we use $\indc{S}$ to denote its indicator function and $|S|$ to denote its cardinality. For a vector $v\in\mathbb{R}^d$, its norms are defined by $\norm{v}_1=\sum_{i=1}^d|v_i|$, $\norm{v}^2=\sum_{i=1}^dv_i^2$ and $\norm{v}_{\infty}=\max_{1\leq i\leq d}|v_i|$. Given two probability distributions $P$ and $Q$ and $\rho>1$, the $\rho$-R\'{e}nyi divergence is defined by $D_{\rho}(P\|Q)=\frac{1}{\rho-1}\log\int\left(\frac{dP}{dQ}\right)^{\rho-1}dP$. The Kullback-Leibler divergence is defined by $D(P\|Q)=\int \log\left(\frac{dP}{dQ}\right)dP$, and the Hellinger distance is defined by $H(P,Q)=\sqrt{\frac{1}{2}\int(\sqrt{dP}-\sqrt{dQ})^2}$. The notation $\mathbb{P}$ and $\mathbb{E}$ are used for generic probability and expectation whose distribution is determined from the context.

\section{Preliminaries on Variational Posterior Convergence}\label{sec:VB}

Assume the observation $X$ is generated from a probability measure $P^*$, and $\wh{Q}$ is the variational posterior distribution defined by (\ref{eq:VB}). It is a fundamental question whether the data-dependent measure $\wh{Q}$ can learn the data generating process $P^*$. The convergence of $\wh{Q}$ was established by \cite{wang2019frequentist} for parametric models. For nonparametric settings, this question has been recently investigated by three independent papers \citep{alquier2020concentration,yang2020alpha,zhang2017convergence}. The main result of this line of work can be summarized as the following theorem.

\begin{thm}[\cite{alquier2020concentration,yang2020alpha,zhang2017convergence}]\label{thm:VB}
Consider a non-negative loss function $L(\cdot,\cdot)$ and a rate $\epsilon_*\geq 1$. Let $C, C_1, C_2, C_3>0$ be constants such that $C>C_2+C_3+2$. We assume
\begin{itemize}
\item For any $\epsilon\geq \epsilon_*$, there exists a set $\mathcal{F}$ and a testing function $\phi$, such that
\begin{equation}
P^*\phi + \sup_{\theta\in\mathcal{F}:L(P^*,P_{\theta})\geq C_1\epsilon^2}P_{\theta}(1-\phi) \leq \exp\left(-C\epsilon^2\right). \label{eq:con-test}
\end{equation}
\item For any $\epsilon \geq \epsilon_*$, the set $\mathcal{F}$ above satisfies
\begin{equation}
\Pi(\mathcal{F}^c)\leq\exp\left(-C\epsilon^2\right). \label{eq:con-sieve}
\end{equation}
\item For some constant $\rho>1$,
\begin{equation}
\Pi\left(D_{\rho}(P^*\|P_{\theta})\leq C_3\epsilon_*^2\right)\geq \exp\left(-C_2\epsilon_*^2\right). \label{eq:con-pm}
\end{equation}
\end{itemize}
Then, for the variational posterior defined in (\ref{eq:VB}), we have
$$P^*\wh{Q}L(P^*,P_{\theta})\leq M(\epsilon_*^2 + \gamma^2),$$
for some constant $M>0$ only depending on $C,C_1$ and $\rho$, where the quantity $\gamma^2$ is defined as
$$\gamma^2=\inf_{Q\in\mathcal{S}}P^*D(Q\|\Pi(\cdot|X)).$$
\end{thm}

The above theorem can be found as Theorem 2.1 in \cite{zhang2017convergence}. Similar conclusions have also been obtained in \cite{alquier2020concentration,yang2020alpha} independently. Theorem \ref{thm:VB} shows that the convergence rate of a variational posterior distribution can be established under almost the same set of prior mass and testing conditions \citep{schwartz1965bayes,barron1988exponential,barron1999consistency,ghosal1999posterior,shen2001rates,ghosal2000convergence,ghosal2007convergence} that lead to the convergence rates of the true posterior $\Pi(\cdot|X)$. The influence of the variational approximation is characterized by the additional term $\gamma^2$ in the error bound. Ideally, one would like to establish the additional inequality $\gamma^2\lesssim \epsilon_*^2$ so that the variational posterior enjoys the same frequentist convergence rate as the true posterior. This goal can be achieved by the following proposition.

\begin{proposition}\label{prop:check-gamma}
Suppose there exists some distribution $Q\in\mathcal{S}$ such that
\begin{eqnarray}
\label{eq:AR1} D(Q\|\Pi) &\leq& C_1'\epsilon_*^2, \\
\label{eq:AR2} QD(P^*\|P_{\theta}) &\leq& C_2'\epsilon_*^2,
\end{eqnarray}
for some constants $C_1',C_2'>0$. Then, we have $\gamma^2\leq (C_1'+C_2')\epsilon_*^2$.
\end{proposition}
The two conditions in Proposition \ref{prop:check-gamma} are first formulated by \cite{alquier2020concentration}, and are slightly extended in \cite{zhang2017convergence}. These two conditions are very easy to work with. As a first example, one can check that that when the variational class $\mathcal{S}$ is the set of all distributions so that $\wh{Q}=\Pi(\cdot|X)$, the two conditions automatically hold and thus the result of Theorem \ref{thm:VB} is reduced to the convergence rate of the true posterior distribution \citep{shen2001rates,ghosal2000convergence,ghosal2007convergence}. To see this, one can use the conditioning method and define a distribution $Q$ by
$$Q(B)=\frac{\Pi(B\cap K)}{\Pi(K)},$$
for any measurable set $B$. The set $K$ is set as the Kullback-Leibler neighborhood,
$$K=\left\{\theta: D(P^*\|P_{\theta})\leq C_3\epsilon_*^2\right\}.$$
It is clear that $Q\in\mathcal{S}$. Since $Q$ is supported on $K$, we must have $QD(P^*\|P_{\theta})\leq C_3\epsilon_*^2$. Moreover, by the condition (\ref{eq:con-pm}) and the fact that $D(P^*\|P_{\theta})\leq D_{\rho}(P^*\|P_{\theta})$ for all $\rho>1$, we have $\Pi(K)\geq\exp(-C_2\epsilon_*^2)$. This leads to the bound $D(Q\|\Pi)=\log\frac{1}{\Pi(K)}\leq C_2\epsilon_*^2$. Hence, (\ref{eq:AR1}) and (\ref{eq:AR2}) hold with $C_1'=C_2$ and $C_2'=C_3$ for the same constants $C_2$ and $C_3$ in (\ref{eq:con-pm}). This example shows that the prior mass condition (\ref{eq:con-pm}) alone suffices to guarantee (\ref{eq:AR1}) and (\ref{eq:AR2}) when $\mathcal{S}$ is the set of all distributions.

When $\mathcal{S}$ is not necessarily the the entire set of distributions, one needs to generalize the prior mass condition (\ref{eq:con-pm}) to accommodate the additional structure imposed by the variational class $\mathcal{S}$. We consider a concrete example of a mean-field class,
$$\mathcal{S}_{\rm MF}=\left\{Q: dQ(\theta)=\prod_{j=1}^ddQ_j(\theta_j)\right\}.$$
The convergence rate of the mean-field variational posterior was established by Theorem 2.4 in \cite{zhang2017convergence}, and we state this result below.
\begin{thm}\label{thm:mf}
Under the same setting of Theorem \ref{thm:VB}, assume the prior distribution satisfies $d\Pi(\theta)=\prod_{j=1}^dd\Pi_j(\theta_j)$. Suppose the conditions (\ref{eq:con-test}), (\ref{eq:con-sieve}) and (\ref{eq:con-pm}) hold. Furthermore, there exists a rectangular subset $\otimes_{j=1}^d\Theta_j\subset\left\{\theta: D(P_{\theta^*}\|P_{\theta})\leq C_4\epsilon_*^2\right\}$ such that
\begin{equation}
\Pi\left(\otimes_{j=1}^d\Theta_j\right) \geq \exp\left(-C_5\epsilon_*^2\right), \label{eq:rec-pm}
\end{equation}
for some constants $C_4,C_5>0$. Then, for the variational posterior defined in (\ref{eq:VB}) with $\mathcal{S}=\mathcal{S}_{\rm MF}$, we have
$$P^*\wh{Q}L(P^*,P_{\theta})\leq M\epsilon_*^2,$$
for some constant $M>0$ only depending on $C,C_1,C_4,C_5$ and $\rho$.
\end{thm}
The condition (\ref{eq:rec-pm}) can be viewed as a generalization of the prior mass condition (\ref{eq:con-pm}). It requires the existence of a \textit{rectangular} Kullback-Leibler neighborhood that receives a not too small prior mass. The rectangular shape is coherent with the product structure imposed by the variational class $\mathcal{S}_{\rm MF}$. Too see why (\ref{eq:rec-pm}) leads to (\ref{eq:AR1}) and (\ref{eq:AR2}), we can still use the conditioning method and define a product distribution $dQ(\theta)=\prod_{j=1}^ddQ_j(\theta_j)$ by
$$Q_j(B)=\frac{\Pi_j(B\cap \Theta_j)}{\Pi_j(\Theta_j)},$$
for any measurable set $B$. It is easy to see that $Q\in\mathcal{S}_{\rm MF}$ and (\ref{eq:AR1}) and (\ref{eq:AR2}) can be directly verified. The existence of the rectangular neighborhood $\otimes_{j=1}^d\Theta_j$ is critical in the construction of the product measure $Q$ above.

Our review of the results on the convergence rates of variational posterior distributions largely follows our previous work \cite{zhang2017convergence}. We also recommend the two concurrent papers \citep{alquier2020concentration,yang2020alpha} to the readers. In addition to the results that are similar to Theorem \ref{thm:VB}, the work \cite{alquier2020concentration} also studied convergence rates under model misspecification, and \cite{yang2020alpha} considered a more general setting that can handle latent variables.

\section{Empirical Bayes and Variational Bayes}\label{sec:EB-VB}

In this section, we will establish a connection between empirical Bayes and variational Bayes. We first review a sieve prior example considered by \cite{zhang2017convergence}. Then, through an appropriate reparametrization of the problem, we show for a general model selection prior, the empirical Bayes posterior can be viewed as a variational Bayes posterior. This connection leads to the convergence rates for general empirical Bayes procedures as long as the hyperparameter set is discrete.

\subsection{Sieve Priors}

Consider a statistical model $\left\{P_{\theta}:\theta\in\otimes_{j=1}^{\infty}\Theta_j\right\}$. We assume that for each coordinate $j$, there is decomposition $\Theta_j=\Theta_{j1}\cup\Theta_{j2}$ and $\Theta_{1j}\cap\Theta_{2j}=\emptyset$. A sieve prior is specified by the following sampling process
\begin{enumerate}
\item Sample an integer $k\sim \pi$;
\item Conditioning on $k$, sample $\theta_j\sim f_{j1}$ for all $j\leq k$, and sample $\theta_j\sim f_{j2}$ for all $j>k$.
\end{enumerate}
As a concrete example, consider $\Theta_{j1}=\mathbb{R}\backslash\{0\}$ and $\Theta_{j2}=\{0\}$. When $f_{1j}$ is Gaussian and $f_{2j}$ is a delta measure at $0$, the sieve prior can be used for Bayesian estimation of a smooth signal in a Sobolev space \citep{shen2001rates,rivoirard2012posterior,gao2016rate}. More generally, we can assume that the densities $f_{j1}$ and $f_{j2}$ satisfy $\int_{\Theta_{j1}}f_{j1}=1$ and $\int_{\Theta_{j2}}f_{j2}=1$. Then, the posterior distribution induced by the sieve prior is
\begin{equation}
d\Pi(\theta|X) = \frac{\sum_k \pi(k) p(X|\theta)\prod_{j\leq k}f_{j1}(\theta_j)\prod_{j>k}f_{j2}(\theta_j)d\theta}{\sum_k \pi(k)\int p(X|\theta)\prod_{j\leq k}f_{j1}(\theta_j)\prod_{j>k}f_{j2}(\theta_j)d\theta}. \label{eq:sieve-post}
\end{equation}
It is observed in \cite{zhang2017convergence} that the mean-field variational approximation of (\ref{eq:sieve-post}) has a form that is very similar to the empirical Bayes procedure. In particular, the empirical Bayes posterior is defined as
\begin{equation}
d\wh{Q}_{\rm EB}(\theta) \propto p(X|\theta)\prod_{j\leq \wh{k}}f_{j1}(\theta_j)\prod_{j>\wh{k}}f_{j2}(\theta_j)d\theta,\label{eq:EB-post-def}
\end{equation}
where $\wh{k}$ is selected according to
$$\wh{k}=\argmax_k\log\left[\pi(k)\int p(X|\theta)\prod_{j\leq k}f_{j1}(\theta_j)\prod_{j>k}f_{j2}(\theta_j)d\theta\right].$$
The following proposition is established as Theorem 5.2 in \cite{zhang2017convergence}.

\begin{proposition}\label{prop:VB-EB-sieve}
Define the following set
\begin{eqnarray*}
\mathcal{S}_{\rm EB} &=& \Bigg\{Q: Q\left(\left(\otimes_{j\leq k}\Theta_{j1}\right)\bigotimes\left(\otimes_{j>k}\Theta_{j2}\right)\right) = 1\text{ for some integer }k
\Bigg\}.
\end{eqnarray*}
Then, the empirical Bayes posterior $\wh{Q}_{\rm EB}$ defined by (\ref{eq:EB-post-def}) is the variational approximation to the hierarchical Bayes posterior (\ref{eq:sieve-post}) in the variational class $\mathcal{S}_{\rm EB}$. In particular, $\wh{Q}_{\rm EB}$ can be written as (\ref{eq:VB}) with $\mathcal{S}=\mathcal{S}_{\rm EB}$ and $\Pi(\cdot|X)$ given by (\ref{eq:sieve-post}).
\end{proposition}

The proof of Proposition \ref{prop:VB-EB-sieve} is straightforward by taking advantage of the sieve prior structure, especially the fact that the supports of $f_{j1}$ and $f_{j2}$ are disjoint for each $j$. Interestingly, there is no requirement for the likelihood function $p(X|\theta)$. However, the result is still restrictive, since it only applies to the class of sieve priors. We will extend Proposition \ref{prop:VB-EB-sieve} to arbitrary priors with a discrete hyperparameter in the next section.

\subsection{Model Selection Priors}

Consider a general statistical model
$$\mathcal{M} = \left\{P_{k,\theta^{(k)}}: k\in\mathcal{K}, \theta^{(k)}\in\Theta^{(k)}\right\},$$
where $\mathcal{K}$ is a countable set and $\Theta^{(k)}$ is some general parameter space. One can think of $\mathcal{K}$ as the set of model index, and we do not require that $\mathcal{K}$ to be an integer set. Given a specific $k$, the likelihood is parameterized by model parameter $\theta^{(k)}$.
A hierarchical prior first sample a model index $k\sim\pi$, and then conditioning on $k$ sample $\theta^{(k)}|k\sim \Pi^{(k)}$. This is arguably the most general prior distribution one can write down with a discrete hyperparameter. It clearly includes the sieve prior as a special case. Furthermore, we allow the likelihood to be parametrized differently with different $k$, and thus the form is even more general than (\ref{eq:basic-lambda-setting}).
The empirical Bayes posterior is given by
$$d\Pi^{(\wh{k})}(\theta^{(\wh{k})}|X)\propto p(X|\wh{k},\theta^{(\wh{k})})d\Pi^{(\wh{k})}(\theta^{(\wh{k})}),$$
with $\wh{k}$ selected according to
\begin{equation}
\wh{k}=\argmax_{k\in\mathcal{K}}\log\left[\pi(k)\int p(X|k,\theta^{(k)})d\Pi^{(k)}(\theta^{(k)})\right]. \label{eq:EB-k-hat-0}
\end{equation}
It turns out to characterize $\Pi^{(\wh{k})}(\cdot|X)$ as a variational approximation is a lot harder than the case of sieve prior. The main difficulty is that the support of $\Pi^{(k)}(\cdot|X)$ is different for each $k\in\mathcal{K}$.

The critical step is to embed different $\theta^{(k)}$'s into a common parameter space. For each $k\in\mathcal{K}$, we define
$$\Xi_k=\left\{\xi=(\theta^{(l)})_{l\in\mathcal{K}}:\theta^{(k)}\in\Theta^{(k)}\text{ and }\theta^{(l)}=*\text{ for all }l\in\mathcal{K}\backslash\{k\}\right\}.$$
Then, define
\begin{eqnarray*}
\Xi &=& \cup_{k\in\mathcal{K}}\Xi_k \\
&=& \left\{\xi=(\theta^{(l)})_{l\in\mathcal{K}}:\theta^{(k)}\in\Theta^{(k)}\text{ for some }k\in\mathcal{K}\text{ and }\theta^{(l)}=*\text{ for all }l\in\mathcal{K}\backslash\{k\}\right\}.
\end{eqnarray*}
The symbol $*$ is used for some arbitrary value outside of any parameter space $\Theta_k$. For any $\xi\in\Xi$, there exists some $k\in\mathcal{K}$ and some $\theta^{(k)}\in\Theta^{(k)}$, such that $\xi=(*,\cdots,*,\theta^{(k)},*,\cdots)$. Moreover, given any pair $(k,\theta^{(k)})$, there corresponds a unique $\xi\in\Xi$. In other words, for each $k\in\mathcal{K}$, there is a bijection between $\Theta_k$ and $\Xi_k$. In this way, we have a common parameter space for different models. With some slight abuse of notation, we can write
$$p(X|k,\theta^{(k)})=p(X|\xi).$$
Given the bijection between $\Theta_k$ and $\Xi_k$, there exists a distribution $\bar{\Pi}^{(k)}$ on $\Xi_k$, such that $\theta^{(k)}\sim \Pi^{(k)}$ is equivalent to $\xi\sim\bar{\Pi}^{(k)}$. The hierarchical model can be equivalently written as $k\sim \pi$ and $\xi|k\sim \bar{\Pi}^{(k)}$. The hierarchical Bayes posterior distribution is
\begin{equation}
d\bar{\Pi}(\xi|X) = \frac{\sum_{k\in\mathcal{K}}\pi(k) p(X|\xi)d\bar{\Pi}^{(k)}(\xi)}{\sum_{k\in\mathcal{K}}\pi(k)\int p(X|\xi)d\bar{\Pi}^{(k)}(\xi)}. \label{eq:HB-gen}
\end{equation}
We can also define an empirical Bayes posterior on the space $\Xi$ by
\begin{equation}
d\wh{Q}_{\rm EB}(\xi) \propto p(X|\xi)d\bar{\Pi}^{(\wh{k})}(\xi), \label{eq:Q-EB-gen}
\end{equation}
where $\wh{k}$ is selected according to
\begin{equation}
\wh{k}=\argmax_{k\in\mathcal{K}}\log\left[\pi(k)\int p(X|\xi)d\bar{\Pi}^{(k)}(\xi)\right]. \label{eq:EB-k-hat}
\end{equation}
Since $\int p(X|k,\theta^{(k)})d\Pi^{(k)}(\theta^{(k)})=\int p(X|\xi)d\bar{\Pi}^{(k)}(\xi)$, the two definitions (\ref{eq:EB-k-hat-0}) and (\ref{eq:EB-k-hat}) are equivalent. Similar to the relation between $\Pi^{(k)}$ and $\bar{\Pi}^{(k)}$, we also know that $\theta^{(\wh{k})}\sim \Pi^{(\wh{k})}(\cdot|X)$ is equivalent to $\xi\sim \wh{Q}_{\rm EB}$.
A variational perspective of the empirical Bayes posterior $\wh{Q}_{\rm EB}$ is given by the following result.

\begin{proposition}\label{prop:VB-EB-general}
Define the following set
$$\mathcal{S}_{\rm EB}=\left\{Q: Q(\Xi_k)=1\text{ for some }k\in\mathcal{K}\right\}$$
Then, the empirical Bayes posterior $\wh{Q}_{\rm EB}$ defined by (\ref{eq:Q-EB-gen}) is the variational approximation to the hierarchical Bayes posterior (\ref{eq:HB-gen}) in the variational class $\mathcal{S}_{\rm EB}$. In particular, $\wh{Q}_{\rm EB}$ can be written as (\ref{eq:VB}) with $\mathcal{S}=\mathcal{S}_{\rm EB}$ and $\Pi(\cdot|X)$ given by (\ref{eq:HB-gen}).
\end{proposition}
\begin{proof}
By the construction of $\Xi_k$, and $\bar{\Pi}_k$, we have for any $k,l\in\mathcal{K}$,
\begin{equation}
\bar{\Pi}^{(l)}(\Xi_k)=\begin{cases}
1, & k=l, \\
0, & k\neq l.
\end{cases} \label{eq:og}
\end{equation}
For any $Q\in\mathcal{S}_{\rm EB}$, there exists some $k\in\mathcal{K}$, such that $Q(\Xi_k)=1$. Then,
\begin{eqnarray}
\nonumber D(Q\|\bar{\Pi}(\cdot|X)) &=& \int_{\Xi_k}\log\frac{dQ(\xi)p_{\bar{\Pi}}(X)}{\sum_{l\in\mathcal{K}}\pi(l) p(X|\xi)d\bar{\Pi}^{(l)}(\xi)}dQ(\xi) \\
\label{eq:fs-z} &=& \int_{\Xi_k}\log\frac{dQ(\xi)}{\pi(k)p(X|\xi)d\bar{\Pi}^{(k)}(\xi)}dQ(\xi) + \log p_{\bar{\Pi}}(X),
\end{eqnarray}
where $p_{\bar{\Pi}}(X)=\sum_{k\in\mathcal{K}}\pi(k) \int p(X|\xi)d\bar{\Pi}^{(k)}(\xi)$ is the marginal distribution. The equality (\ref{eq:fs-z}) is derived from the property (\ref{eq:og}). Therefore, for each specific $k\in\mathcal{K}$, $D(Q\|\bar{\Pi}(\cdot|X))$ is minimized by
$$dQ^{(k)}(\xi)\propto p(X|\xi)d\bar{\Pi}^{(k)}(\xi),$$
among all $Q$ such that $Q(\Xi_k)=1$ holds. In other words, we have $Q^{(k)}=\bar{\Pi}^{(k)}(\cdot|X)$. Plug $Q^{(k)}=\bar{\Pi}^{(k)}(\cdot|X)$ back into the objective function $D(Q\|\bar{\Pi}(\cdot|X))$, and we can see that $\wh{Q}_{\rm EB}=\bar{\Pi}^{(\wh{k})}(\cdot|X)=\argmin_{Q\in\mathcal{S}_{\rm EB}}D(Q\|\bar{\Pi}(\cdot|X))$, and $\wh{k}$ is determined by (\ref{eq:EB-k-hat}).
\end{proof}

The result of Proposition \ref{prop:VB-EB-general} is very general, thanks to the embedding of the original parameter spaces into a common $\Xi$. It basically covers all empirical Bayes procedures with a discrete hyperparameter. There is no assumption on the likelihood and the prior. Perhaps the only assumption is the hyperparameter weight $\pi(k)$ used in (\ref{eq:EB-k-hat-0}) or (\ref{eq:EB-k-hat}) needs to be a probability distribution. In fact, even this condition can be relaxed. By scrutinizing the proof of Proposition \ref{prop:VB-EB-general}, all we need is that hierarchical Bayes posterior (\ref{eq:HB-gen}) is well defined. This means the conclusion of Proposition \ref{prop:VB-EB-general} continues to hold with a more general
$$\wh{k}=\argmax_{k\in\mathcal{K}}\log\left[w(k)\int p(X|\xi)d\bar{\Pi}^{(k)}(\xi)\right],$$
as long as $\sum_{k\in\mathcal{K}}w(k) \int p(X|\xi)d\bar{\Pi}^{(k)}(\xi)<\infty$ holds almost surely.

Proposition \ref{prop:VB-EB-general} is more general than Proposition \ref{prop:VB-EB-sieve}. In fact, when specialized to the setting of sieve priors, Proposition \ref{prop:VB-EB-general} removes the assumption $\Theta_{1j}\cap\Theta_{2j}=\emptyset$ required by Proposition \ref{prop:VB-EB-sieve} for the sieve prior. This is because the target of the variational approximation of Proposition \ref{prop:VB-EB-general} is a hierarchical Bayes posterior on the space $\Xi$, where the distributions $\{\bar{\Pi}^{(k)}\}_{k\in\mathcal{K}}$ naturally satisfy the orthogonality condition (\ref{eq:og}). In comparison, Proposition \ref{prop:VB-EB-sieve} works with a variational approximation to a hierarchical Bayes posterior on the original parameter space. The condition $\Theta_{1j}\cap\Theta_{2j}=\emptyset$ is then necessary to guarantee the orthogonality of $\{{\Pi}^{(k)}\}_{k\in\mathcal{K}}$.

Since $\wh{Q}_{\rm EB}$ is a variational posterior according to Proposition \ref{prop:VB-EB-general}, we can then derive its convergence rate using Theorem \ref{thm:VB}.
\begin{thm}\label{thm:EB-ms}
Consider a non-negative loss function $L(\cdot,\cdot)$ and a rate $\epsilon_*\geq 1$. Let $C, C_1, C_2, C_3>0$ be constants such that $C>C_2+C_3+2$. We assume
\begin{itemize}
\item For any $\epsilon\geq \epsilon_*$, there exist subsets $\bar{\mathcal{K}}\subseteq\mathcal{K}$, $\mathcal{F}_k\subseteq\Theta^{(k)}$ and a testing function $\phi$, such that
\begin{equation}
P^*\phi + \sup_{\{k\in\bar{\mathcal{K}},\theta^{(k)}\in\mathcal{F}_k:L(P^*,P_{k,\theta^{(k)}})\geq C_1\epsilon^2\}}P_{k, \theta^{(k)}}(1-\phi) \leq \exp\left(-C\epsilon^2\right). \label{eq:EB-con-test}
\end{equation}
\item For any $\epsilon \geq \epsilon_*$, the subsets $\bar{\mathcal{K}}$ and $\{\mathcal{F}_k\}_{k\in\mathcal{K}}$ defined above satisfy
\begin{equation}
\sum_{k\not\in\bar{\mathcal{K}}}\pi(k)+\sum_{k\in\bar{\mathcal{K}}}\pi(k)\Pi^{(k)}(\mathcal{F}_k^c)\leq\exp\left(-C\epsilon^2\right). \label{eq:EB-con-sieve}
\end{equation}
\item There exists some $k^*\in\mathcal{K}$ and some constant $\rho>1$ such that
\begin{equation}
\pi(k^*)\Pi^{(k^*)}\left(\left\{\theta^{(k^*)}\in\Theta^{(k^*)}: D_{\rho}(P^*\|P_{k^*,\theta^{(k^*)}})\leq C_3\epsilon_*^2\right\}\right) \geq\exp(-C_2\epsilon_*^2). \label{eq:EB-pm-0}
\end{equation}
\end{itemize}
Then, for $\wh{k}$ defined by (\ref{eq:EB-k-hat-0}), we have
$$P^*\int L(P^*,P_{\wh{k},\theta^{(\wh{k})}})d\Pi^{(\wh{k})}(\theta^{(\wh{k})}|X)\leq M\epsilon_*^2,$$
for some $M>0$ only depending on $C,C_1$ and $\rho$.
\end{thm}

\begin{proof}
By Proposition \ref{prop:VB-EB-general}, we can view $\wh{Q}_{\rm EB}=\bar{\Pi}^{(\wh{k})}(\cdot|X)$ as the minimizer of $D(Q\|\bar{\Pi}(\cdot|X))$ under the constraint $Q\in\mathcal{S}_{\rm EB}$. Therefore, we can directly apply Theorem \ref{thm:VB}. According to the bijection between $\Theta_k$ and $\Xi_k$ and the relation between $\Pi^{(k)}$ and $\bar\Pi^{(k)}$, (\ref{eq:EB-con-test}) and (\ref{eq:EB-con-sieve}) are equivalent to 
$$
P^*\phi +\sup_{\xi\in\mathcal{F}:L(P^*,P_{\xi})\geq C_1\epsilon^2}P_{\xi}(1-\phi) \leq \exp\left(-C\epsilon^2\right),
$$
and
$$
\bar\Pi(\mathcal{F}^c)\leq\exp\left(-C\epsilon^2\right).
$$
with $\mathcal{F} = \cup_{k\in\bar{\mathcal{K}}}\left\{\xi = (*,\cdots,*,\theta^{(k)},*,\cdots): \theta^{(k)}\in\mathcal{F}_k\right\}$ and $\bar\Pi = \sum_{k\in\mathcal{K}}\pi(k)\bar\Pi^{(k)}$. Thus, (\ref{eq:con-test}) and (\ref{eq:con-sieve}) in Theorem \ref{thm:VB} are satisfied and it is sufficient to bound $\gamma^2$ by checking the two conditions (\ref{eq:AR1}) and (\ref{eq:AR2}) of Proposition \ref{prop:check-gamma}. To do this, we define a distribution $Q$ by
$$Q(B)=\frac{\bar{\Pi}(B\cap K)}{\bar{\Pi}(K)},$$
with $\bar{\Pi}=\sum_{k\in\mathcal{K}}\pi(k)\bar{\Pi}^{(k)}$ and
$$K=\left\{\xi\in\Xi_{k^*}: D_{\rho}(P^*\|P_{\xi})\leq C_3\epsilon_*^2\right\}.$$
This construction implies that $Q(\Xi_{k^*})=1$ and thus $Q\in\mathcal{S}_{\rm EB}$. Since $Q$ is supported on $K$, we clearly have $QD(P^*\|P_{\xi})\leq QD_{\rho}(P^*\|P_{\xi})\leq C_3\epsilon_*^2$. By (\ref{eq:EB-pm-0}), we also have
$$D(Q\|\bar{\Pi})=\log\frac{1}{\bar{\Pi}(K)}=\log\frac{1}{\pi(k^*)\Pi^{(k^*)}\left(\left\{\theta^{(k^*)}\in\Theta^{(k^*)}: D_{\rho}(P^*\|P_{k^*,\theta^{(k^*)}})\leq C_3\epsilon_*^2\right\}\right)}\leq C_2\epsilon_*^2.$$
Hence, (\ref{eq:AR1}) and (\ref{eq:AR2}) hold with $C_1'=C_2$ and $C_2'=C_3$. Finally, note that (\ref{eq:EB-pm-0}) also implies (\ref{eq:con-pm}), and thus by the conclusion of Theorem \ref{thm:VB}, we have
$$P^*\Pi^{(\wh{k})}\left(L(P^*,P_{\wh{k},\theta^{(\wh{k})}})\Big|X\right)=P^*\bar{\Pi}^{(\wh{k})}\left(L(P^*,P_{\xi})\Big|X\right)\lesssim \epsilon_*^2,$$
as desired.
\end{proof}

Theorem \ref{thm:EB-ms} derives a convergence rate for the empirical Bayes posterior. With the variational approximation perspective, the proof of Theorem \ref{thm:EB-ms} is almost straightforward. Note that all the conditions of Theorem \ref{thm:EB-ms} are stated with respect to the original parameter space. The space $\Xi$ and the corresponding variational approximation are only used in the proof. The form of Theorem \ref{thm:EB-ms} is very similar to that of Theorem \ref{thm:mf}. Both Theorem \ref{thm:EB-ms} and Theorem \ref{thm:mf} generalize the prior mass condition in \cite{schwartz1965bayes,barron1988exponential,barron1999consistency,ghosal1999posterior,shen2001rates,ghosal2000convergence,ghosal2007convergence} to accommodate the structure of the variational classes. A sufficient condition for (\ref{eq:EB-pm-0}) is the existence of $k^*\in\mathcal{K}$ such that
\begin{eqnarray}
\label{eq:EB-pm-model} \pi(k^*) &\geq& \exp(-C_2'\epsilon_*^2), \\
\label{eq:EB-pm-parameter} \Pi^{(k^*)}\left(\left\{\theta^{(k^*)}\in\Theta^{(k^*)}: D_{\rho}(P^*\|P_{k^*,\theta^{(k^*)}})\leq C_3\epsilon_*^2\right\}\right) &\geq& \exp(-C_2''\epsilon_*^2),
\end{eqnarray}
with $C_2'+C_2''=C_2$. That is, there exists a model, such that both the prior probability of this model and the prior probability of the information neighborhood of $P^*$ within the model are not too small. In fact, conditions similar to (\ref{eq:EB-pm-model}) and (\ref{eq:EB-pm-parameter}) are already found in the literature of hierarchical Bayes convergence rates \citep{rousseau2010rates,rivoirard2012posterior,han2017bayes}. This suggests that many nonparametric Bayesian estimation problems that are solved in the literature by hierarchical Bayes have the same theoretical guarantees with empirical Bayes procedures.

\subsection{Some Examples}

In this section, we illustrate the results of Theorem \ref{thm:EB-ms} with two examples of nonparametric estimation.

\paragraph{Infinite dimensional exponential families.} Define a probability measure $P_{\theta}$ by
$$\frac{dP_{\theta}}{d\ell}=\exp\left(\sum_{j=0}^{\infty}\theta_j\phi_j-c(\theta)\right),$$
where $\ell$ denotes the Lebesgue measure on $[0,1]$, and $\phi_j$ is the $j$th Fourier basis function of $L^2[0,1]$, and $c(\theta)$ is given by
$$c(\theta)=\log\int_0^1\exp\left(\sum_{j=0}^{\infty}\theta_j\phi_j(x)\right)dx.$$
Since $\phi_0(x ) = 1$ and $\theta_0$ can take an arbitrary value without changing $P_{\theta}$, we simply set $\theta_0=0$. In other words, $P_{\theta}$ is fully parameterized by $\theta=(\theta_1,\theta_2,...)$.
Given i.i.d. observations from the product measure $P_{\theta^*}^n$, our goal is to estimate $P_{\theta^*}$, where $\theta^*$ is assumed to belong to the Sobolev ball,
$$\Theta_{\alpha}(R)=\left\{\theta=(\theta_j)_{j=1}^{\infty}:\sum_{j=1}^{\infty}j^{2\alpha}\theta_j^2\leq R^2\right\}.$$
The smoothness parameter $\alpha$ and the radius $R$ are assumed to be constants throughout the section.

For any integer $k$, consider the prior distribution $\Pi^{(k)}=\left(\otimes_{j\leq k}N(0,\sigma^2)\right)\bigotimes \left(\otimes_{j>k}\delta_0\right)$. That is, to sample $\theta\sim\Pi^{(k)}$, one first sample independent $\theta_j\sim N(0,\sigma^2)$ for all $j\leq k$ and then set $\theta_j=0$ for all $j>0$. This leads to the prior distribution
$$d\Pi^{(k)}(\theta|X_1,\cdots,X_n)\propto \prod_{i=1}^np(X_i|\theta)d\Pi^{(k)}(\theta).$$
We select $k$ via the empirical Bayes principle. That is,
$$
\wh{k}=\argmax_{k\in[n]}\log\left[w(k)\int \prod_{i=1}^np(X_i|\theta)d\Pi^{(k)}(\theta)\right],
$$
where the weight function is chosen to be proportional to the probability mass function of Poisson distribution, $w(k)=\tau^k/k!$. We then obtain the empirical Bayes posterior $\Pi^{(\wh{k})}(\cdot|X_1,\cdots,X_n)$.
\begin{thm}\label{thm:den-exp-EB}
Consider the above empirical Bayes posterior distribution defined by $\Pi^{(k)}$ and $w(k)$ with some constants $\sigma^2,\tau>0$. Assume $\alpha>1/2$. We then have
$$P_{\theta^*}^n\int H^2(P_{\theta},P_{\theta^*})d\Pi^{(\wh{k})}(\theta|X_1,\cdots,X_n)\leq Mn^{-\frac{2\alpha}{2\alpha+1}}(\log n)^{\frac{2\alpha}{2\alpha+1}},$$
for some constant $M>0$ uniformly over all $\theta^*\in\Theta_{\alpha}(R)$.
\end{thm}

Adaptive Bayesian density estimation with infinite exponential family approximation has been studied by \cite{scricciolo2006convergence,rivoirard2012posterior}. In particular, it was shown by \cite{rivoirard2012posterior} that with the additional prior distribution $k\sim \text{Poisson}(\tau)$ on the hyperparameter, the hierarchical Bayes posterior achieves the near optimal convergence rate. The work \cite{zhang2017convergence} shows that a Gaussian mean-field variational approximation to this hierarchical Bayes posterior can achieve the same rate. Theorem \ref{thm:den-exp-EB} complements the results of \cite{rivoirard2012posterior,zhang2017convergence} by showing that the same theoretical guarantee can be established by the empirical Bayes posterior as well.

The result of Theorem \ref{thm:den-exp-EB} can be easily derived from Theorem \ref{thm:EB-ms}. In fact, thanks to the familiar prior mass condition (\ref{eq:EB-pm-0}), the proof of Theorem \ref{thm:EB-ms} directly follows the arguments used in \cite{rivoirard2012posterior,zhang2017convergence}. There is no need to develop any new technical tool to prove the result for the empirical Bayes posterior!
We also remark that the the specific choice of $N(0,\sigma^2)$ and $\text{Poisson}(\tau)$ in our prior construction can be easily replaced by more general class of distributions as considered in \cite{rivoirard2012posterior,zhang2017convergence}. We omit such an extension for the simplicity of presentation.

\paragraph{Density estimation via location-scale mixtures.}  Our second example considers Bayesian density estimation via location-scale mixture models. The location-scale mixture density is defined as
\begin{equation}\label{eq:mixture}
p(x|k,\theta^{(k)}) = \sum_{j=1}^kw_j\psi_\sigma(x-\mu_j),
\end{equation}
where $k\in\mathbb{N}_+$, $\theta^{(k)} = (\mu, w,\sigma)$ with $\sigma>0$, $\mu = (\mu_1,\cdots,\mu_k)\in\R^k$, $w = (w_1,\cdots, w_k)\in\Delta_k = \left\{w\in\R^k: w_j\geq 0\mbox{ for $1\leq j\leq k$ and }\sum_{j=1}^kw_j = 1\right\}$ and
\begin{equation}\label{eq:kernel}
\psi_\sigma(x) = \frac{1}{2\sigma\Gamma\left(1+\frac{1}{p}\right)}\exp(-(|x|/\sigma)^p),
\end{equation} 
for some positive even integer $p$. The kernel $\psi_{\sigma}(\cdot)$ has a pre-specified form, for example, Gaussian density when $p=2$, while the parameters $k$ and $\theta^{(k)} = (w,\mu,\sigma)$ are to be learned from the data.

Given i.i.d. observations $X_1,...,X_n$ sampled from some density function $f^*$, our goal is to estimate the density $f^*$ through the location-scale mixture model (\ref{eq:mixture}). We denote the probability distribution of the mixture density $p(x|k,\theta^{(k)})$ as $P_{k,\theta^{(k)}}$ and a probability distribution with a general density $f$ as $P_{f}$. In \cite{kruijer2010adaptive}, a hierarchical Bayes procedure is proposed and a nearly minimax optimal convergence rate is derived for the posterior distribution. The recent work \cite{zhang2017convergence} shows the same theoretical result can be obtained by two different variational approximations to the hierarchical Bayes posterior distribution, one with a mean-field variational class and the other involving an additional approximation through latent variables of the clustering labels.

We will construct an empirical Bayes procedure to achieve the same theoretical result by applying Theorem \ref{thm:EB-ms}. For any integer $k$, consider a prior $\theta^{(k)} = (\mu, w,\sigma)\sim\Pi^{(k)}$ by sampling $\mu_j\sim N(0,\sigma_0^2)$, $w\sim\rm{Dir}(\alpha_0,\alpha_0,\cdots,\alpha_0)$ with $\alpha_0<1$ and $\tau\sim\Gamma(a_0, b_0)$ independently. This leads to the posterior distribution
$$d\Pi^{(k)}(\theta^{(k)}|X_1,\cdots,X_n)\propto \prod_{i=1}^np(X_i|k,\theta^{(k)})d\Pi^{(k)}(\theta^{(k)}).$$
The number of clusters $k$ is selected according to
$$\wh{k}=\argmax_{k\in[n]}\log\left[w(k)\int \prod_{i=1}^np(X_i|k,\theta^{(k)})d\Pi^{(k)}(\theta^{(k)})\right],$$
where the weight function is chosen to be proportional to the probability mass function of Poisson distribution, $w(k)=\xi_0^k/k!$. Again, we remark that more general prior distributions on $\theta^{(k)} = (\mu, w,\sigma)$ and more general weight function used in \cite{kruijer2010adaptive,zhang2017convergence} can also be considered.

Next, we list the conditions on the true density function $f^*$:
\begin{enumerate}[label=B\arabic*]
\item \label{eq:B1}(Smoothness) The logarithmic density function $\log f^*$ is assumed to be locally $\alpha$-H\"older smooth. In other words, for the derivative $l_j(x) = \frac{d^j}{dx^j}\log f^*(x)$, there exists a polynomial $L(\cdot)$ and a constant $\gamma>0$ such that,
\begin{equation}\label{eq:true-condition1}
|l_{\floor{\alpha}}(x)-l_{\floor{\alpha}}(y)|\leq L(x)|x-y|^{\alpha-\floor{\alpha}},
\end{equation}
for all $x,y$ that satisfies $|x-y|\leq\gamma$. Here, the degree and the coefficients of the polynomial $L(\cdot)$ are all assumed to be constants.
Moreover, the derivative $l_j(x)$ satisfies the bound $\int |l_j(x)|^{\frac{2\alpha+\epsilon}{j}}f^*(x)dx<s_{\max}$ for all $j=1,...,\floor{\alpha}$ with  some constants $\epsilon, s_{\max}>0$.
\item \label{eq:B2}(Tail) There exist positive constants $T$, $\xi_1$, $\xi_2$, $\xi_3$ such that
\begin{equation}\label{eq:true-condition3}
f^*(x)\leq \xi_1e^{-\xi_2|x|^{\xi_3}},
\end{equation}
for all $|x|\geq T$.
\item \label{eq:B3}(Monotonicity) There exist constants $x_m<x_M$ such that $f^*$ is nondecreasing on $(-\infty, x_m)$ and is nonincreasing on $(x_M,\infty)$. Without loss of generality, we assume $f^*(x_m) = f^*(x_M) = c$ and $f^*(x)\geq c$ for all $x_m<x<x_M$ with some constant $c>0$.
\end{enumerate}
These conditions are exactly the same as in \cite{kruijer2010adaptive}. The conditions allow a well-behaved approximation to the true density by a location-scale mixture.

\begin{thm}\label{thm:den-mixture-EB}
Consider the above empirical Bayes posterior distribution defined by $\Pi^{(k)}$ and $w(k)$ with some constants $\xi_0,\sigma_0, \alpha_0, a_0, b_0>0$. Assume the density function $f^*$ satisfies (\ref{eq:B1})-(\ref{eq:B3}). We then have
$$P_{f^*}^n\int H^2(P_{\wh{k},\theta^{(\wh{k})}},P_{f^*})d\Pi^{(\wh{k})}(\theta^{(\wh{k})}|X_1,\cdots,X_n)\leq Mn^{-\frac{2\alpha}{2\alpha+1}}(\log n)^{\frac{2\alpha r}{\alpha+1}},$$
for some constant $M>0$, where $r=\frac{p}{\min\{p,\xi_3\}}+\max\{1,\frac{2}{\min\{p,\xi_3\}}\}$, with $p$ and $\xi_3$ defined in (\ref{eq:kernel}) and (\ref{eq:true-condition3}).
\end{thm}

Theorem \ref{thm:den-mixture-EB} shows that the empirical Bayes posterior distributions achieves the near minimax rate of estimating a density function that is H\"older smooth. According to Proposition \ref{prop:VB-EB-general}, this result can be viewed as the third variational approximation, in addition to the previous two variational approximations considered in \cite{zhang2017convergence}, to the hierarchical Bayes posterior in \cite{kruijer2010adaptive} that enjoys the same theoretical guarantee.

\section{A General Analysis of Empirical Bayes Posterior} \label{sec:gen}

In this section, we study the theoretical properties of empirical Bayes posterior distributions when the hyperparameter set $\Lambda$ is not necessarily discrete. One way to deal with a general $\Lambda$ is discretization. This is the technique used by \cite{rousseau2017asymptotic}, but it leads to empirical process and entropy conditions that may not be easy to verify. We introduce a new ``prior decomposition" technique, and we will demonstrate its applications through various high-dimensional estimation problems.

\subsection{A General Theorem}

Consider a general prior distribution $\theta\sim \Pi_{\lambda}$ that is supported on $\Theta$ and indexed by some hyperparameter $\lambda\in\Lambda$. The posterior distribution given some $\lambda\in\Lambda$ is defined by
$$\Pi_{\lambda}(B|X)=\frac{\int_B p(X|\theta)d\Pi_{\lambda}(\theta)}{\int p(X|\theta)d\Pi_{\lambda}(\theta)},$$
for any measurable set $B$.
Following the empirical Bayes principle, we define the empirical Bayes posterior distribution $\Pi_{\wh{\lambda}}(\cdot|X)$ by selecting $\wh{\lambda}$ via
\begin{equation}
\wh{\lambda}=\argmax_{\lambda\in\Lambda}\log\left[w(\lambda)\int p(X|\theta)d\Pi_{\lambda}(\theta)\right]. \label{eq:hat-lambda}
\end{equation}
Our goal in this section is to establish convergence rates of the empirical Bayes procedures that allow both continuous hyperparameter set $\Lambda$ and unbounded parameter space $\Theta$. We assume there exists a discrete collection of subspaces $\{\Theta_{Z}:Z\in\mathcal{Z}\}$ of $\Theta$, such that $\Pi_{\lambda}$ admits the following decomposition
\begin{equation}
\Pi_{\lambda}=\sum_{Z\in\mathcal{Z}}\nu_{\lambda}(Z)\Gamma_Z, \label{eq:prior-decomp}
\end{equation}
where for each $Z\in\mathcal{Z}$, $\Gamma_Z$ is a probability measure on the subspace $\Theta_Z$. The sequence $\{\nu_{\lambda}(Z):Z\in\mathcal{Z}\}$ is a discrete probability so that $\sum_{Z\in\mathcal{Z}}\nu_{\lambda}(Z)=1$. 

The idea behind (\ref{eq:prior-decomp}) is simple. For almost all nonparametric and high-dimensional models, there exist some underlying low-dimensional structures. The collection of such low-dimensional structures is usually a discrete set. Then, the decomposition (\ref{eq:prior-decomp}) can be understood as a two-step sampling process of $\Pi_{\lambda}$. One first sample some $Z\in\mathcal{Z}$ according to $\nu_{\lambda}(Z)$, and then given $Z$, sample $\theta|Z\sim \Gamma_Z$. This idea is fundamentally different from discretizing the hyperparameter set $\Lambda$. Instead, we choose to directly work with the underlying discrete low-dimensional structures of statistical models. We emphasize that the uniqueness of the decomposition (\ref{eq:prior-decomp}) is not important at all. We only require the existence of some decomposition in the form of (\ref{eq:prior-decomp}) such that appropriate conditions on $\nu_{\lambda}(Z)$ and $\Gamma_Z$ are satisfied.

For each $Z\in\mathcal{Z}$, the decomposition (\ref{eq:prior-decomp}) naturally induces the following quantity,
\begin{equation}
\gamma(Z)=\max_{\lambda\in\Lambda}\left[w(\lambda)\nu_{\lambda}(Z)\right]. \label{eq:effective-weight}
\end{equation}
We call $\gamma(Z)$ the effective weight on the structure $Z$. We are now well prepared to state the following theorem.

\begin{thm}\label{thm:EB-con}
Consider a non-negative loss function $L(\cdot,\cdot)$. Let $C,C_1,C_2,C_3,C_4,C_5>0$ be constants such that $C> 2C_5$. Assume there exist a rate function $\epsilon:Z\mapsto \epsilon(Z)$, $\lambda^*\in\Lambda$ and $Z^*\in\mathcal{Z}$ such that $\theta^*\in\Theta_{Z^*}$, $\epsilon(Z^*)\geq 1$, and the following conditions hold:
\begin{itemize}
\item There exists a testing function $\phi$, such that for any $\epsilon^2\geq \epsilon(Z^*)^2$ and any $Z\in\mathcal{Z}$,
\begin{align}
\nonumber P_{\theta^*}\phi &\leq \exp(-C_1\epsilon(Z^*)^2), \\
\label{eq:EB-test} \sup_{\theta\in\Theta_Z: L(\theta,\theta^*)\geq\epsilon^2} P_{\theta}(1-\phi) & \leq \exp\left(-C\epsilon^2 + C_2\left(\epsilon(Z)^2 + \epsilon(Z^*)^2\right)\right).
\end{align}
\item There exists some map $\delta:\mathcal{Z}\rightarrow\mathbb{R}_+$, such that
\begin{equation}
\sum_{Z\in\mathcal{Z}}\frac{\gamma(Z)\delta(Z)}{w(\lambda^*)\nu_{\lambda^*}(Z^*)\delta(Z^*)}\exp\left(2C_2\epsilon(Z)^2\right) \leq \exp(C_4\epsilon(Z^*)^2). \label{eq:EB-sieve}
\end{equation}
\item For some constant $\rho>1$ and the map $\delta:\mathcal{Z}\rightarrow\mathbb{R}_+$ above,
\begin{equation}
\frac{\Gamma_Z\left(\left\{\theta\in\Theta_Z: L(\theta,\theta^*)\leq \epsilon^2\right\}\right)}{\Gamma_{Z^*}\left(\left\{\theta\in\Theta_{Z^*}: D_{\rho}(P_{\theta^*}\|P_{\theta})\leq C_3\epsilon(Z^*)^2\right\}\right)} \leq \frac{\delta(Z)}{\delta(Z^*)}\exp\left(C_5\epsilon^2+C_2\left(\epsilon(Z)^2+\epsilon(Z^*)^2\right)\right), \label{eq:EB-pm}
\end{equation}
for any $\epsilon^2\geq \epsilon(Z^*)^2$ and $Z\in\mathcal{Z}$.
\end{itemize}
Then, for $\wh{\lambda}$ defined by (\ref{eq:hat-lambda}), we have
$$P_{\theta^*}\Pi_{\wh{\lambda}}\left(L(\theta,\theta^*)>M\epsilon(Z^*)^2\Big|X\right)\leq 4\exp(-C'\epsilon(Z^*)^2),$$
for some constants $M,C'>0$ only depending on $C,C_1,C_2,C_3,C_4,C_5$ and $\rho$.
\end{thm}
\begin{proof}
Define
$$H=\left\{\int\frac{p(X|\theta)}{p(X|\theta^*)}\geq\exp\left(-(C_3+1)\epsilon(Z^*)^2\right)\Gamma_{Z^*}(K)\right\},$$
with $K=\left\{\theta\in\Theta_{Z^*}: D_{\rho}(P_{\theta^*}\|P_{\theta})\leq C_3\epsilon(Z^*)^2\right\}$. Then, for the testing function $\phi$ that satisfies (\ref{eq:EB-test}), we have
\begin{eqnarray}
\nonumber && P_{\theta^*}\Pi_{\wh{\lambda}}\left(L(\theta,\theta^*)>M\epsilon(Z^*)^2|X\right) \\
\label{eq:3-terms} &\leq& P_{\theta^*}\phi + P_{\theta^*}(H^c) + P_{\theta^*}\frac{\int_{L(\theta,\theta^*)>M\epsilon(Z^*)^2} p(X|\theta)d\Pi_{\wh{\lambda}}(\theta)}{\int p(X|\theta)d\Pi_{\wh{\lambda}}(\theta)}(1-\phi)\indc{H}.
\end{eqnarray}
By (\ref{eq:EB-test}), we have $P_{\theta^*}\phi \leq \exp(-C_1\epsilon(Z^*)^2)$. To bound $P_{\theta^*}(H^c)$, we introduce a probability measure $\wt{\Gamma}_{Z^*}$, defined by
$$\wt{\Gamma}_{Z^*}(B)=\frac{\Gamma_{Z^*}(B\cap K)}{\Gamma_{Z^*}(K)},$$
for any measurable set $B$. Then,
\begin{eqnarray*}
P_{\theta^*}(H^c) &=& \mathbb{P}_{\theta^*}\left(\int\frac{p(X|\theta)}{p(X|\theta^*)}<\exp\left(-(C_3+1)\epsilon(Z^*)^2\right)\Gamma_{Z^*}(K)\right) \\
&\leq& \mathbb{P}_{\theta^*}\left(\int_K\frac{p(X|\theta)}{p(X|\theta^*)}<\exp\left(-(C_3+1)\epsilon(Z^*)^2\right)\Gamma_{Z^*}(K)\right) \\
&\leq& \exp(-(C_3+1)(\rho-1)\epsilon(Z^*)^2)\mathbb{E}_{\theta^*}\left(\int_K\frac{p(X|\theta)}{p(X|\theta^*)}d\wt{\Gamma}_{Z^*}(\theta)\right)^{-(\rho-1)} \\
&\leq& \exp(-(C_3+1)(\rho-1)\epsilon(Z^*)^2)\int_K\exp\left((\rho-1) D_{\rho}(P_{\theta^*}\|P_{\theta})\right)d\wt{\Gamma}_{Z^*}(\theta) \\
&\leq& \exp\left(-(\rho-1) \epsilon(Z^*)^2\right),
\end{eqnarray*}
where the last inequality uses the definition of $K$. Now we consider the last term of (\ref{eq:3-terms}). By the definition of $\wh{\lambda}$ in (\ref{eq:hat-lambda}), we have
\begin{eqnarray}
\nonumber && P_{\theta^*}\frac{\int_{L(\theta,\theta^*)>M\epsilon(Z^*)^2} p(X|\theta)d\Pi_{\wh{\lambda}}(\theta)}{\int p(X|\theta)d\Pi_{\wh{\lambda}}(\theta)}(1-\phi)\indc{H} \\
\nonumber &=& P_{\theta^*}\frac{\sum_{Z\in\mathcal{Z}}w(\wh{\lambda})\nu_{\wh{\lambda}}(Z)\int_{L(\theta,\theta^*)>M\epsilon(Z^*)^2} \frac{p(X|\theta)}{p(X|\theta^*)}d\Gamma_{Z}(\theta)}{w(\wh{\lambda})\int \frac{p(X|\theta)}{p(X|\theta^*)}d\Pi_{\wh{\lambda}}(\theta)}(1-\phi)\indc{H} \\
\label{eq:prelim-b-3rd} &\leq& P_{\theta^*}\frac{\sum_{Z\in\mathcal{Z}}\gamma(Z)\int_{L(\theta,\theta^*)>M\epsilon(Z^*)^2} \frac{p(X|\theta)}{p(X|\theta^*)}d\Gamma_{Z}(\theta)}{w(\lambda^*)\int \frac{p(X|\theta)}{p(X|\theta^*)}d\Pi_{\lambda^*}(\theta)}(1-\phi)\indc{H}.
\end{eqnarray}
When the event $H$ holds, the denominator of (\ref{eq:prelim-b-3rd}) can be lower bounded by
\begin{eqnarray*}
w(\lambda^*)\int \frac{p(X|\theta)}{p(X|\theta^*)}d\Pi_{\lambda^*}(\theta) &\geq& w(\lambda^*)\nu_{\lambda^*}(Z^*)\int \frac{p(X|\theta)}{p(X|\theta^*)}d\Gamma_{Z^*}(\theta) \\
&\geq& w(\lambda^*)\nu_{\lambda^*}(Z^*)\exp\left(-(C_3+1)\epsilon(Z^*)^2\right)\Gamma_{Z^*}(K).
\end{eqnarray*}
Therefore,
\begin{eqnarray*}
&& P_{\theta^*}\frac{\int_{L(\theta,\theta^*)>M\epsilon(Z^*)^2} p(X|\theta)d\Pi_{\wh{\lambda}}(\theta)}{\int p(X|\theta)d\Pi_{\wh{\lambda}}(\theta)}(1-\phi)\indc{H} \\
&\leq& \exp((C_3+1)\epsilon(Z^*)^2)\sum_{Z\in\mathcal{Z}}\frac{\gamma(Z)}{w(\lambda^*)\nu_{\lambda^*}(Z^*)}\frac{1}{\Gamma_{Z^*}(K)}\int_{L(\theta,\theta^*)>M\epsilon(Z^*)^2}P_{\theta}(1-\phi)d\Gamma_Z(\theta).
\end{eqnarray*}
Define $R_l(Z)=\left\{\theta\in\Theta_Z: lM\epsilon(Z^*)^2<L(\theta,\theta^*)\leq (l+1)M\epsilon(Z^*)^2\right\}$. Then, we have
\begin{eqnarray}
\nonumber && \frac{1}{\Gamma_{Z^*}(K)}\int_{L(\theta,\theta^*)>M\epsilon(Z^*)^2}P_{\theta}(1-\phi)d\Gamma_Z(\theta) \\
\nonumber &=& \frac{1}{\Gamma_{Z^*}(K)}\sum_{l=1}^{\infty}\int_{R_l(Z)}P_{\theta}(1-\phi)d\Gamma_Z(\theta) \\
\nonumber &\leq& \sum_{l=1}^{\infty}\frac{\Gamma_Z(R_l(Z))}{\Gamma_{Z^*}(K)}\sup_{\theta\in R_l(Z)}P_{\theta}(1-\phi) \\
\label{eq:pad1} &\leq& \exp(2C_2(\epsilon(Z^*)^2+\epsilon(Z)^2))\frac{\delta(Z)}{\delta(Z^*)}\sum_{l=1}^{\infty}\exp\left(-(C-2C_5)lM\epsilon(Z^*)^2\right) \\
\label{eq:pad2} &\leq& 2\exp\left(-\left((C-2C_5)M-2C_2\right)\epsilon(Z^*)^2 + 2C_2\epsilon(Z)^2\right)\frac{\delta(Z)}{\delta(Z^*)} \\
\label{eq:pad3} &\leq& 2\exp\left(-\frac{(C-2C_5)M}{2}\epsilon(Z^*)^2 + 2C_2\epsilon(Z)^2\right)\frac{\delta(Z)}{\delta(Z^*)}
\end{eqnarray}
The inequality (\ref{eq:pad1}) uses (\ref{eq:EB-test}) and (\ref{eq:EB-pm}). The next two bounds (\ref{eq:pad2}) and (\ref{eq:pad3}) are by $C> 2C_5$ and $M\geq \frac{4C_2}{C-2C_5}$. Having obtained the bound (\ref{eq:pad3}), we then have
\begin{eqnarray*}
&& P_{\theta^*}\frac{\int_{L(\theta,\theta^*)>M\epsilon(Z^*)^2} p(X|\theta)d\Pi_{\wh{\lambda}}(\theta)}{\int p(X|\theta)d\Pi_{\wh{\lambda}}(\theta)}(1-\phi)\indc{H} \\
&\leq& 2\exp\left(-\left(\frac{M(C-2C_5)}{2}-C_3-1\right)\epsilon(Z^*)^2\right)\sum_{Z\in\mathcal{Z}}\frac{\gamma(Z)}{w(\lambda^*)\nu_{\lambda^*}(Z^*)}\frac{\delta(Z)}{\delta(Z^*)}\exp(2C_2\epsilon(Z)^2) \\
&\leq& 2\exp\left(-\left(\frac{M(C-2C_5)}{2}-C_3-C_4-1\right)\epsilon(Z^*)^2\right),
\end{eqnarray*}
where the last inequality is by (\ref{eq:EB-sieve}). Set $M$ to be sufficiently large, and we obtain the desired conclusion that (\ref{eq:3-terms}) is bounded by $4\exp(-C'\epsilon(Z^*)^2)$ with some $C'>0$.
\end{proof}

Theorem \ref{thm:EB-con} derives a high-probability convergence rate for $\Pi_{\wh{\lambda}}(\cdot|X)$. A bound for $P_{\theta^*}\int L(\theta,\theta^*)d\Pi_{\wh{\lambda}}(\theta|X)$ with the same rate can also be derived by integrating up the tail probability. The three conditions (\ref{eq:EB-test})-(\ref{eq:EB-pm}), though complicated, are actually quite easy to work with. In fact, the three conditions (\ref{eq:EB-test})-(\ref{eq:EB-pm}) directly correspond to the three standard ``prior mass and testing" conditions (\ref{eq:con-test})-(\ref{eq:con-pm}). Condition (\ref{eq:EB-test}) guarantees the existence of a test that is adaptive to the underlying structure. Condition (\ref{eq:EB-sieve}) plays the same role as (\ref{eq:con-sieve}) that controls the complexity of the prior distribution. Finally, we have a prior mass ratio condition in (\ref{eq:EB-pm}). The prior mass ratio bound is a standard way to deal with unbounded parameter space $\Theta$ in the literature \citep{ghosal2000convergence,castillo2012needles}.

Verifying (\ref{eq:EB-test})-(\ref{eq:EB-pm}) is just some direct calculations. The main trick is to choose appropriate rate function $\epsilon(Z)$ and the map $\delta(Z)$ that are natural to the problem. One usually choose $\epsilon(Z)^2$ to be the minimax rate of the problem. On the other hand, choosing $\delta(Z)$ is a more subtle issue. For most applications, the naive choice $\delta(Z)=1$ suffices. However, we will show in Section \ref{sec:slm} that sometimes a smart choice of $\delta(Z)$ can be crucial to the result.

\subsection{Applications to Spike-and-Slab Priors}

A leading example of continuous hyperparameter is the spike-and-slab prior \citep{mitchell1988bayesian,george1993variable} that models sparse vectors. Given some $\lambda\in[n]$, the spike-and-slab prior on $\mathbb{R}^p$ is defined by
\begin{equation}
\Pi_{\lambda}=\bigotimes_{j=1}^p\left((1-\lambda)\delta_0+\lambda G\right), \label{eq:spike-slab-def}
\end{equation}
where $\delta_0$ is a delta measure at $0$, and $G$ is some slab distribution. In this paper, we consider $G$ to be the Laplace distribution with density function $g(x)=\frac{\tau}{2}e^{-\tau|x|}$. For $\theta\sim \Pi_{\lambda}$, an equivalent sampling process is to first sample independent latent variables $z_1,\cdots,z_p\sim\text{Bernoulli}(\lambda)$, and then sample $\theta_j|z_j\sim (1-z_j)\delta_0+z_jG$ for each $j\in[p]$. From this perspective, we can equivalently write the prior distribution as
$$\Pi_{\lambda}=\sum_{S\subset[p]}(1-\lambda)^{p-|S|}\lambda^{|S|}\left(\bigotimes_{j\in S^c}\delta_0\right) \otimes \left(\bigotimes_{j\in S}G\right),$$
which is a natural prior decomposition in the form of (\ref{eq:prior-decomp}). To formally put $\Pi_{\lambda}$ in the framework of (\ref{eq:prior-decomp}), we let $Z=S$, and define $\Theta_S=\{\theta\in\mathbb{R}^p:\theta_j=0\text{ for all }j\notin S\}$. Then, we can write $\Pi_{\lambda}$ as (\ref{eq:prior-decomp}) with
\begin{equation}
\Gamma_S=\left(\bigotimes_{j\in S^c}\delta_0\right) \otimes \left(\bigotimes_{j\in S}G\right), \label{eq:Gamma-spike-slab}
\end{equation}
and
$$\nu_{\lambda}(S)=(1-\lambda)^{p-|S|}\lambda^{|S|}.$$
To select $\lambda$, the most popular method in hierarchical Bayes is to further sample $\lambda$ from a conjugate beta prior. We therefore use the weight function
\begin{equation}
w(\lambda)=\lambda^{\alpha-1}(1-\lambda)^{\beta-1}, \label{eq:spike-and-slab-weight}
\end{equation}
for our empirical Bayes procedure. Then, it is easy to calculate the effective weight on $S$,
$$\gamma(S)=\max_{\lambda\in [0,1]}\left[\nu_{\lambda}(S)w(\lambda)\right]=\left(\frac{\alpha+|S|-1}{p+\alpha+\beta-2}\right)^{\alpha+|S|-1}\left(\frac{p-|S|+\beta-1}{p+\alpha+\beta-2}\right)^{p-|S|+\beta-1}.$$
Let us first verify the condition (\ref{eq:EB-sieve}), since this condition is independent of the likelihood function.
\begin{lemma}\label{lem:spsl-sum}
Assume $\frac{\alpha}{e(p+\beta-1)}\geq p^{-C_1'}$ and $\frac{e(\alpha+p)}{\beta-1}\leq p^{-C_2'}$ for some sufficiently large constants $C_1'>C_2'>0$. There exists some constant $C_4>0$, such that for any $S^*\subset[p]$, there is some $\lambda^*\in[0,1]$, such that
$$\sum_{S\subset[p]}\frac{\gamma(S)}{w(\lambda^*)\nu_{\lambda^*}(S^*)}\exp\left(2C_2|S|\log p\right)\leq \exp\left(C_4|S^*|\log p\right).$$
\end{lemma}
Lemma \ref{lem:spsl-sum} shows that the condition (\ref{eq:EB-sieve}) holds with $\epsilon(S)^2=|S|\log p$ and $\delta(S)=1$. The requirement on $\alpha$ and $\beta$ can be easily satisfied. For example, one can choose $\alpha=1$ and $\beta = p^C$ a sufficiently large degree $C>0$.
Next, we will verify the other two conditions (\ref{eq:EB-test}) and (\ref{eq:EB-pm}) in the context of Gaussian sequence model and sparse linear regression.

\paragraph{Gaussian sequence model.}

The first example that we apply the spike-and-slab prior is the sparse sequence model with data generating process $Y\sim P_{\theta^*}=N(\theta^*,I_p)$. We assume the vector $\theta^*\in\mathbb{R}^p$ is sparse and thus there exists some subset $S^*\subset[p]$ such that $\theta^*\in\Theta_{S^*}$. We use the empirical Bayes posterior $\Pi_{\wh{\lambda}}(\cdot|Y)$ for statistical estimation, where $\Pi_{\lambda}(\cdot|Y)$ is the posterior distribution induced by the spike-and-slab prior (\ref{eq:spike-slab-def}), and $\wh{\lambda}$ is selected according to (\ref{eq:hat-lambda}) with weight (\ref{eq:spike-and-slab-weight}).

Having verified the condition (\ref{eq:EB-sieve}) by Lemma \ref{lem:spsl-sum}, we need to check the remaining two conditions (\ref{eq:EB-test}) and (\ref{eq:EB-pm}) in the setting of Gaussian sequence model. We use the loss function $L(\theta,\theta^*)=\|\theta-\theta^*\|^2$ and set $\rho=2$ so that we also have $D_{\rho}(P_{\theta^*}\|P_{\theta})=\|\theta-\theta^*\|^2$. Recall that for any subset $S\subset[p]$, we have $\epsilon(S)=|S|\log p$ and $\delta(S)=1$.

For the testing condition (\ref{eq:EB-test}), we first define
$$\phi_S=\indc{\|Y-\theta^*\|_{S\cup S^*}^2 > 6\left(|S|\log p + |S^*|\log p\right)},$$
where we use the notation $\|v\|_S^2=\sum_{j\in S}v_j^2$. The overall test is given by
\begin{equation}
\phi=\max_{S\subset[p]}\phi_S. \label{eq:test-sequence}
\end{equation}
The following testing error bound is straightforward by a union bound argument and a standard chi-squared deviation inequality.
\begin{lemma}\label{lem:spsl-seq-test}
Assume $n\geq 8$.
Under the setting of Gaussian sequence model, the testing procedure (\ref{eq:test-sequence}) satisfies
\begin{eqnarray*}
P_{\theta^*}\phi &\leq& \exp(-|S^*|\log p), \\
\sup_{\theta\in\Theta_S:\|\theta-\theta^*\|^2\geq\epsilon^2}P_{\theta}(1-\phi) &\leq& \exp\left(-\frac{2}{3}\epsilon^2 + 5(|S|\log p + |S^*|\log p)\right),
\end{eqnarray*}
for any $\epsilon^2\geq |S^*|\log p$ and any $S\subset[p]$.
\end{lemma}

Next, we check the prior mass ratio condition (\ref{eq:EB-pm}). Recall that for any subset $S\subset[p]$, $\Gamma_S$ is a product measure on $\Theta_S$. Its definition is given by (\ref{eq:Gamma-spike-slab}) with $G$ being the Laplace distribution with density function $g(x)=\frac{\tau}{2}e^{-\tau|x|}$.
\begin{lemma}\label{lem:spsl-seq-pm}
For any $\tau>0$, there exists some constant $C_2>0$ only depending on $\tau$, such that
$$\frac{\Gamma_S\left(\left\{\theta\in\Theta_S:\|\theta-\theta^*\|^2\leq \epsilon^2\right\}\right)}{\Gamma_{S^*}\left(\left\{\theta\in\Theta_{S^*}:\|\theta-\theta^*\|^2\leq |S^*|\log p\right\}\right)} \leq \exp\left(\frac{1}{6}\epsilon^2 + C_2\left(|S|\log p+|S^*|\log p\right)\right),$$
for any $\epsilon^2\geq |S^*|\log p$ and any $S\subset[p]$.
\end{lemma}

Lemma \ref{lem:spsl-sum}, Lemma \ref{lem:spsl-seq-test} and Lemma \ref{lem:spsl-seq-pm} together imply that the three conditions of Theorem \ref{thm:EB-con} are satisfied with $C=\frac{2}{3}$, $C_1=1$, $C_3=1$, $C_5=\frac{1}{6}$ and some $C_2,C_4$ depending on $\tau$. Thus, all conditions in Theorem \ref{thm:EB-con} are satisfied. 

\begin{thm}\label{thm:EB-sequence}
Consider the empirical Bayes procedure defined with a Laplace slab for some constant $\tau>0$ and weight function (\ref{eq:spike-and-slab-weight}) that satisfies $p^{C_1'}\geq\beta/\alpha\geq p^{C_2'}$ for some sufficiently large constants $C_1'>C_2'>0$. Then, the conditions (\ref{eq:EB-test})-(\ref{eq:EB-pm}) are satisfied under the setting of Gaussian sequence model $P_{\theta}=N(\theta,I_p)$. As a consequence, we have
$$P_{\theta^*}\Pi_{\wh{\lambda}}\left(\|\theta-\theta^*\|^2 > M|S^*|\log p\Big|Y\right)\leq 4\exp(-C'|S^*|\log p),$$
with some constants $M,C'>0$ uniformly over all $\theta^*\in\Theta_{S^*}$.
\end{thm}

Previous results on empirical Bayes procedures using spike-and-slab priors under the setting of Gaussian sequence model are given by \cite{johnstone2004needles,castillo2018empirical}. The seminal work \cite{johnstone2004needles} shows that the posterior median achieves the minimax rate as long as the slab has a tail that is at least as heavy as the Laplace distribution. Interestingly, \cite{castillo2018empirical} shows that the Laplace slab does not lead to optimal convergence rate for the entire empirical Bayes posterior distribution, and one has to use a Cauchy slab for this purpose. We emphasize that the lower bound of the Laplace slab in \cite{castillo2018empirical} is proved with the weight function $w(\lambda)=\indc{\lambda\in[t,1]}$ for some $t=o(1)$. Theorem \ref{thm:EB-sequence} thus complements the result of \cite{castillo2018empirical} by showing that the Laplace slab can still lead to optimal rate for the entire empirical Bayes posterior distribution, as long as a different weight function (\ref{eq:spike-and-slab-weight}) is used with parameters $\alpha$ and $\beta$ set appropriately.

\paragraph{Sparse linear regression.} Consider a regression model $Y\sim P_{\theta^*}=N(X\theta^*,I_n)$ with some design matrix $X\in\mathbb{R}^{n\times p}$. Again, we assume a sparse $\theta^*$ and there exists some $S^*\subset[p]$ such that $\theta^*\in\Theta_{S^*}$. The empirical Bayes posterior $\wh{\Pi}_{\wh{\lambda}}(\cdot|Y)$ is defined in the same way as the Gaussian sequence model with the likelihood replaced by $P_{\theta}=N(X\theta,I_n)$.

We will check the conditions (\ref{eq:EB-test}) and (\ref{eq:EB-pm}) with the loss function $L(\theta,\theta^*)=\|X(\theta-\theta^*)\|^2$, and for $\rho=2$, we have $D_{\rho}(P_{\theta^*}\|P_{\theta})=\|X(\theta-\theta^*)\|^2$. Recall that for any subset $S\subset[p]$, we have $\epsilon(S)=|S|\log p$ and $\delta(S)=1$.

In order to introduce the testing procedure, we first introduce some additional notation. For any subset $S\subset[p]$, we use $P_{S}\in\mathbb{R}^{n\times n}$ for the projection matrix onto the subspace spanned by the columns of $X$ in the set $S$. To be more precise, we have $P_S=X_S(X_S^TX_S)^-X_S^T$, where $X_S\in\mathbb{R}^{n\times |S|}$ is a submatrix of $X$ that collects the columns of $X$ in the set $S$, and $(X_S^TX_S)^-$ is the generalized inverse of $X_S^TX_S$. For any $S\subset[p]$, define
$$\phi_S=\indc{\|P_{S\cup S^*}(Y-X\theta^*)\|^2> 6(|S|\log p+|S^*|\log p)},$$
and the overall test is given by
\begin{equation}
\phi=\max_{S\subset[p]}\phi_S. \label{eq:test-regression}
\end{equation}
The analysis of the testing error is the same as that of (\ref{eq:test-sequence}), and we get the same conclusion.
\begin{lemma}\label{lem:spsl-reg-test}
Assume $p\geq 8$. Under the setting of sparse linear regression, the testing procedure (\ref{eq:test-regression}) satisfies
\begin{eqnarray*}
P_{\theta^*}\phi &\leq& \exp(-|S^*|\log p), \\
\sup_{\theta\in\Theta_S:\|\theta-\theta^*\|^2\geq\epsilon^2}P_{\theta}(1-\phi) &\leq& \exp\left(-\frac{2}{3}\epsilon^2 + 5(|S|\log p + |S^*|\log p)\right),
\end{eqnarray*}
for any $\epsilon^2\geq |S^*|\log p$ and any $S\subset[p]$.
\end{lemma}

Before establishing the condition (\ref{eq:EB-pm}), we need to define a new quantity. For the design matrix $X\in\mathbb{R}^{n\times p}$, we write $\|X\|=\max_{j\in[p]}\|X_j\|$, where $X_j$ is the $j$th column of $X$. Given a $S\subset[p]$, the compatibility number of of $X$ is defined by
$$\kappa(S)=\inf\left\{\frac{\|Xu\||S|^{1/2}}{\|X\|\|u_S\|_1}: \|u_{S^c}\|_1\leq 3\|u_S\|_1, u_S\neq 0\right\},$$
where $u_S$ is a subvector of $u$ with coordinates in $S$.
The quantity $\kappa(S)$ is commonly used in the literature of Lasso \citep{van2007deterministic,bickel2009simultaneous,buhlmann2011statistics}, and has also been used in the context of Bayesian sparse linear regression \citep{castillo2015bayesian}.

\begin{lemma}\label{lem:spsl-reg-pm}
Set $\tau=\bar{\tau}\|X\|$ with $\bar{\tau}=p^{-\zeta}$ for some constant $\zeta>0$. As long as $\kappa(S^*)\geq p^{-\zeta}$, there exists some constant $C_2>0$ only depending on $\zeta$, such that
$$\frac{\Gamma_S\left(\left\{\theta\in\Theta_S: \|X(\theta-\theta^*)\|^2\leq\epsilon^2\right\}\right)}{\Gamma_{S^*}\left(\left\{\theta\in\Theta_{S^*}:\|X(\theta-\theta^*)\|^2\leq |S^*|\log p\right\}\right)}\leq \exp\left(\bar{\tau}\epsilon^2 + C_2(|S|\log p+ |S^*|\log p)\right),$$
for any $\epsilon^2\geq |S^*|\log p$ and any $S\subset[p]$.
\end{lemma}

Lemma \ref{lem:spsl-sum}, Lemma \ref{lem:spsl-reg-test} and Lemma \ref{lem:spsl-reg-pm} together imply that the three conditions of Theorem \ref{thm:EB-con} are satisfied for some constants $C,C_1,C_2,C_3,C_4,C_5>0$. The requirement $C>2C_5$ is certainly satisfied, since $\bar{\tau}=p^{-\zeta}$ can be arbitrarily small for a sufficiently large $p$. We can then immediately write down a theorem for the convergence rate of the empirical Bayes posterior by applying Theorem \ref{thm:EB-con} to the current setting.

\begin{thm}\label{eq:EB-reg}
Consider the empirical Bayes procedure defined with a Laplace slab for $\tau=p^{-\zeta}\|X\|$ with some constant $\zeta>0$ and weight function (\ref{eq:spike-and-slab-weight}) that satisfies $p^{C_1'}\geq\beta/\alpha\geq p^{C_2'}$ for some sufficiently large constant $C_1'>C_2'>0$. As long as $\kappa(S^*)\geq p^{-\zeta}$, the conditions (\ref{eq:EB-test})-(\ref{eq:EB-pm}) are satisfied under the setting of sparse linear regression $P_{\theta}=N(X\theta,I_n)$. As a consequence, we have
$$P_{\theta^*}\Pi_{\wh{\lambda}}\left(\|X(\theta-\theta^*)\|^2 > M|S^*|\log p\Big|Y\right)\leq 4\exp(-C'|S^*|\log p),$$
with some constants $M,C'>0$ uniformly over all $\theta^*\in\Theta_{S^*}$.
\end{thm}

Frequentist convergence rates of Bayesian sparse linear regression with spike-and-slab prior have be studied by \cite{castillo2015bayesian}. The convergence rate of a variational approximation has recently been investigated by \cite{ray2019variational}. Theorem \ref{eq:EB-reg} provides the first result on the convergence rate of the empirical Bayes posterior for this problem. Together with \cite{castillo2015bayesian,ray2019variational}, we conclude that the optimal convergence rate for sparse linear regression can be achieved by any of the hierarchical Bayes, variational Bayes, and empirical Bayes methods.

On a technical side, our result requires the condition $\kappa(S^*)\geq p^{-\zeta}$ to achieve the rate $|S^*|\log p$. Given that the constant $\zeta$ can be arbitrary, this condition is very weak. In comparison, to achieve the same rate, Lasso requires $\kappa(S^*)\gtrsim 1$ \citep{van2007deterministic,bickel2009simultaneous,buhlmann2011statistics}. The results of \cite{castillo2015bayesian} and \cite{ray2019variational} also require $\kappa(S^*)\gtrsim 1$ for both the hierarchical Bayes and the variational Bayes procedures. While the condition $\kappa(S^*)\gtrsim 1$ is necessary for Lasso \citep{zhang2014lower}, we believe the same assumption in \cite{castillo2015bayesian} and \cite{ray2019variational} can be replaced by $\kappa(S^*)\geq p^{-\zeta}$ with an improved analysis.

\subsection{Applications to Structured Linear Models}\label{sec:slm}

In this section, we give another application of Theorem \ref{thm:EB-con}. We show the entire framework of structured linear models in \cite{gao2015general} can be solved by empirical Bayes procedures. This makes Theorem \ref{thm:EB-con} applicable to every special case of the framework including biclustering, regression with group sparsity, dictionary learning, etc.

We first review the framework of structured linear models in \cite{gao2015general}. A random vector $Y\in\mathbb{R}^N$ follows a structured linear model if $Y\sim N(\X_Z(B),I_N)$, where the notation $\X_Z$ represents a linear operator. The mean vector of the Gaussian distribution $\X_Z(B)$ has two elements, a model parameter $B$ and a structure $Z$ that indexes the linear operator $\X_Z(\cdot)$. The structure $Z$ is an element of some discrete space $\Z_{\lambda}$, which is further indexed by $\lambda\in\Lambda$ for some finite set $\Lambda$. We introduce a function $\ell(\Z_{\lambda})$ to denote the dimension of the parameter $B$. In other words, we have $B\in\mathbb{R}^{\ell(\Z_{\lambda})}$, and $\ell(\Z_{\lambda})$ is referred to as the intrinsic dimension of the parameter $B$. Thus, $\X_Z$ is a linear operator from $\mathbb{R}^{\ell(\Z_{\lambda})}$ to $\mathbb{R}^N$, and it can be viewed as a matrix in $\mathbb{R}^{N\times \ell(\Z_{\lambda})}$.

The most straightforward example in this framework is the sparse linear regression model. In this example, the sparse regression vector can be decomposed as $\beta=(\beta_S^T,0_{S^c}^T)^T$ for some subset $S$. Then, we can write $X\beta=X_S\beta_S=\X_Z(B)$ by letting $B=\beta_S$, $Z=S$ and $\X_Z(\cdot)=X_S$, where $X_S$ is a submatrix of $X$ with columns in $S$.

A hierarchical Bayes procedure is proposed by \cite{gao2015general} for the structured linear model. The prior distribution $\theta\sim\Pi$ can be described by the following sampling process:
\begin{enumerate}
\item Sample $\lambda\sim \pi$ from $\Lambda$. The distribution $\pi$ is specified by the probability mass function $\pi(\lambda)\propto \frac{\Gamma(\ell(\Z_{\lambda}))}{\Gamma(\ell(\Z_{\lambda})/2)}\exp(-D\epsilon(\Z_{\lambda})^2)$, where $\epsilon(\Z_{\lambda})^2=\ell(\Z_{\lambda})+\log|\ell(\Z_{\lambda})|$;
\item Conditioning on $\lambda$, sample $Z$ uniformly from the set $\bar{\Z}_{\lambda}=\{Z\in\Z_{\lambda}: \det(\X_Z^T\X_Z)>0\}$;
\item Conditioning on $(\lambda,Z)$, sample $B\sim f_{\ell(\Z_{\lambda}),\X_Z,\tau}$ and set $\theta=\X_Z(B)$, where $$f_{\ell(\Z_{\lambda}),\X_Z,\tau}(B)\propto \exp(-\tau\|\X_Z(B)\|)$$ is an elliptical Laplace distribution on $\mathbb{R}^{\ell(\Z_{\lambda})}$.
\end{enumerate}
We give some remarks on the above prior distribution. First, the quantity $\epsilon(\Z_{\lambda})^2=\ell(\Z_{\lambda})+\log|\ell(\Z_{\lambda})|$ plays the role of the complexity of the model. Similar definitions have also appeared in the frequentist literature \citep{barron1999risk,birge2001gaussian,yang2000combining}. In many cases, $\epsilon(\Z_{\lambda})^2$ can be shown to be the minimax rate of the problem \citep{klopp2019structured}. Though the most nature prior on $\lambda$ would be the complexity prior \citep{castillo2012needles} that has the form $\pi(\lambda)\propto \exp(-D\epsilon(\Z_{\lambda})^2)$, it is important to include the extra factor of Gamma function ratio $\frac{\Gamma(\ell(\Z_{\lambda}))}{\Gamma(\ell(\Z_{\lambda})/2)}$. This is because the elliptical Laplace distribution used in Step 3 of the sampling process has density function
\begin{equation}
f_{\ell(\Z_{\lambda}),\X_Z,\tau}(B)=\frac{\sqrt{\det(\X_Z^T\X_Z)}}{2}\left(\frac{\tau}{\sqrt{\pi}}\right)^{\ell(\Z_{\lambda})}\frac{\Gamma(\ell(\Z_{\lambda})/2)}{\Gamma(\ell(\Z_{\lambda}))}\exp(-\tau\|\X_Z(B)\|).\label{eq:def-e-lap}
\end{equation}
It turns out the normalizing constant $\frac{\Gamma(\ell(\Z_{\lambda})/2)}{\Gamma(\ell(\Z_{\lambda}))}$ of the density has an effect on the model selection that cannot be neglected. The extra factor $\frac{\Gamma(\ell(\Z_{\lambda}))}{\Gamma(\ell(\Z_{\lambda})/2)}$ in $\pi(\lambda)$ thus corrects this unwanted effect. Second, the sampling of $Z$ is only from the subset $\bar{\Z}_{\lambda}$ of non-degenerate structures. There is no need to sample $Z$ with $\det(\X_Z^T\X_Z)=0$, since any $\X_Z(B)$ can be equivalently written as $\X_{\bar{Z}}(\bar{B})$ for some $\bar{Z}$ such that $\det(\X_{\bar{Z}}^T\X_{\bar{Z}})>0$. Moreover, the requirement $\det(\X_Z^T\X_Z)>0$ ensures that $f_{\ell(\Z_{\lambda}),\X_Z,\tau}$ is a proper non-degenerate density function on $\mathbb{R}^{\ell(\Z_{\lambda})}$. Last but not least, the choice of the elliptical Laplace distribution helps to deal with model parameters that are unbounded due to its exponential tail. A heavier-tailed elliptical distribution would also work here.

With the Gaussian likelihood $P_{\theta}=N(\theta,I_N)$ and the prior $\theta\sim\Pi$ that is specified by the above sampling process, one can define the posterior distribution by $d\Pi(\theta|Y)\propto p(X|\theta)d\Pi(\theta)$.  Assume
\begin{equation}
|\{\lambda\in\Lambda: t-1<\epsilon(\Z_{\lambda})^2\leq t\}| \leq t\text{ for all }t\in\mathbb{N}. \label{eq:G-vdV-Z}
\end{equation}
It has been shown by \cite{gao2015general} that for any $Z^*\in\bar{\Z}_{\lambda^*}$, any $B^*\in\mathbb{R}^{\ell(\Z_{\lambda^*})}$ and any $\lambda^*\in\Lambda$, the posterior distribution satisfies
\begin{equation}
P_{\X_{Z^*}(B^*)}\Pi\left(\|\theta-\X_{Z^*}(B^*)\|^2>M\epsilon(\Z_{\lambda^*})^2\Big|Y\right) \leq \exp(-C'\epsilon(\Z_{\lambda^*})^2), \label{eq:slm-hb-rate}
\end{equation}
for some constants $M,C'>0$. The result (\ref{eq:slm-hb-rate}) is then applied to various high-dimensional estimation problems to achieve optimal posterior contraction.

In this section, we show an analogous result with $\lambda$ selected via the empirical Bayes principle can be derived from Theorem \ref{thm:EB-con}. Given a $\lambda\in\Lambda$, we use the notation $\Pi_{\lambda}$ for the distribution that consists of Step 2 and Step 3 in the sampling process of $\Pi$. In other words, to sample $\theta\sim\Pi_{\lambda}$, we first sample $Z\sim\text{Uniform}(\bar{\Z}_{\lambda})$, and then set $\theta=\X_Z(B)$ with $B\sim f_{\ell(\Z_{\lambda}),\X_Z,\tau}$. This leads to the posterior distribution $d\Pi_{\lambda}(\theta|Y)\propto p(X|\theta)d\Pi_{\lambda}(\theta)$ for every $\lambda\in\Lambda$. The empirical Bayes posterior is $\Pi_{\wh{\lambda}}(\cdot|Y)$, where $\wh{\lambda}$ is selected according to
\begin{equation}
\wh{\lambda}=\argmax_{\lambda\in\Lambda}\log\left[w(\lambda)\int p(Y|\theta)d\Pi_{\lambda}(\theta)\right]. \label{eq:fancy-weight}
\end{equation}
To achieve the same theoretical performance as the hierarchical Bayes, the weight function is given by
$$w(\lambda)=\frac{\Gamma(\ell(\Z_{\lambda}))}{\Gamma(\ell(\Z_{\lambda})/2)}\exp(-D\epsilon(\Z_{\lambda})^2),$$
which plays the same role as $\pi$ in the hierarchical Bayes procedure.

In the framework of structured linear models, the distribution $\Pi_{\lambda}$ admits a very natural prior decomposition required by Theorem \ref{thm:EB-con}. Without loss of generality, we assume that the sets $\{\bar{\Z}_{\lambda}\}_{\lambda\in\Lambda}$ are mutually disjoint. According to the two-step sampling process of $\Pi_{\lambda}$, we can then write
$$\Pi_{\lambda}=\sum_{Z\in\mathcal{Z}}\nu_{\lambda}(Z)\Gamma_Z,$$
where $\Z=\cup_{\lambda\in\Lambda}\bar{\Z}_{\lambda}$ and $\nu_{\lambda}(Z)=\frac{1}{|\bar{\Z}_{\lambda}|}\indc{Z\in\bar{\Z}_{\lambda}}$. Given the form of $\nu_{\lambda}(Z)$, we can equivalently write
$$\Pi_{\lambda}=\frac{1}{|\bar{\Z}_{\lambda}|}\sum_{Z\in\bar{\Z}_{\lambda}}\Gamma_Z.$$
For any $Z\in\bar{\Z_{\lambda}}$, $\Gamma_Z$ is the distribution of $\theta=\X_Z(B)$ with $B\sim f_{\ell(\Z_{\lambda}),\X_Z,\tau}$. Clearly, $\Gamma_Z$ is supported on the subspace $\Theta_Z$, defined by
$$\Theta_Z=\left\{\theta\in\mathbb{R}^N: \theta=\X_Z(B)\text{ for some }B\in\mathbb{R}^{\ell(\Z_{\lambda})}\right\}.$$
The effective weight (\ref{eq:effective-weight}) can be easily calculated. It is given by
$$\gamma(Z)=\max_{\lambda'\in\Lambda}\left[w(\lambda')\frac{1}{|\bar{\Z}_{\lambda'}|}\indc{Z\in\bar{\Z}_{\lambda'}}\right]=w(\lambda)\frac{1}{|\bar{\Z}_{\lambda}|}=\frac{\Gamma(\ell(\Z_{\lambda}))}{\Gamma(\ell(\Z_{\lambda})/2)}\exp(-D\epsilon(\Z_{\lambda})^2)\frac{1}{|\bar{\Z}_{\lambda}|},$$
for any $Z\in\bar{\Z_{\lambda}}$.

To check the three conditions (\ref{eq:EB-test})-(\ref{eq:EB-pm}), we need to specify $\epsilon(Z)$ and $\delta(Z)$. For any $Z\in\bar{\mathcal{Z}}_{\lambda}$, we take $\epsilon(Z)=\epsilon(\Z_{\lambda})$ and $\delta(Z)=\frac{\Gamma(\ell(\Z_{\lambda})/2)}{\Gamma(\ell(\Z_{\lambda}))}$. We first check (\ref{eq:EB-pm}). By the specific values of $w(\lambda)$, $\nu_{\lambda}(Z)$, $\gamma(Z)$, $\delta(Z)$ and $\epsilon(Z)$, we have for any $\lambda^*\in\Lambda$ and any $Z^*\in\bar{\Z}_{\lambda^*}$,
\begin{eqnarray}
\nonumber && \sum_{Z\in\mathcal{Z}}\frac{\gamma(Z)\delta(Z)}{w(\lambda^*)\nu_{\lambda^*}(Z^*)\delta(Z^*)}\exp\left(2C_2\epsilon(Z)^2\right) \\
\nonumber &=& \sum_{\lambda\in\Lambda}\sum_{Z\in\bar{\Z}_{\lambda}}\frac{w(\lambda)|\bar{\Z}_{\lambda^*}|\delta(Z)}{w(\lambda^*)|\bar{\Z}_{\lambda}|\delta(Z^*)}\exp\left(2C_2\epsilon(\Z_{\lambda})^2\right) \\
\label{eq:cancel-Gamma} &=& \sum_{\lambda\in\Lambda}\sum_{Z\in\bar{\Z}_{\lambda}}\frac{|\bar{\Z}_{\lambda^*}|}{|\bar{\Z}_{\lambda}|}\exp\left(2C_2\epsilon(\Z_{\lambda})^2-D\epsilon(\Z_{\lambda})^2+D\epsilon(\Z_{\lambda^*})^2\right) \\
\nonumber &\leq& \exp((D+1)\epsilon(\Z_{\lambda^*})^2)\sum_{\lambda\in\Lambda}\exp\left(-(D-2C_2)\epsilon(\Z_{\lambda})^2\right) \\
\nonumber &\leq& \exp((D+1)\epsilon(\Z_{\lambda^*})^2+1),
\end{eqnarray}
as long as $D\geq 2C_2+1$. The last inequality uses the fact that $\sum_{\lambda\in\Lambda}\exp\left(-\epsilon(\Z_{\lambda})^2\right)\leq e$ under the condition (\ref{eq:G-vdV-Z}) (see (\ref{eq:sum-to-int}) in Section \ref{sec:pf}). Therefore, the condition is satisfied, and we formulate this result as the following lemma.
\begin{lemma}\label{lem:slm-sieve}
Assume (\ref{eq:G-vdV-Z}) and $D\geq 2C_2+1$. Then, there exists a constant $C_4>0$ only depending on $D$, such that for any $\lambda^*\in\Lambda$ and any $Z^*\in\bar{\Z}_{\lambda^*}$, we have
$$\sum_{Z\in\mathcal{Z}}\frac{\gamma(Z)\delta(Z)}{w(\lambda^*)\nu_{\lambda^*}(Z^*)\delta(Z^*)}\exp\left(2C_2\epsilon(Z)^2\right)\leq \exp(C_4\epsilon(Z^*)^2).$$
\end{lemma}
The choice of $\delta(Z)=\frac{\Gamma(\ell(\Z_{\lambda})/2)}{\Gamma(\ell(\Z_{\lambda}))}$ is crucial for the cancellation of the gamma functions in the equality (\ref{eq:cancel-Gamma}). This is the benefit due to the flexibility of the conditions of Theorem \ref{thm:EB-con} that allow for a map $\delta(Z)$ that is not necessarily $1$.

Next, we verify the condition (\ref{eq:EB-test}) with the loss function $L(\theta,\theta^*)=\|\theta-\theta^*\|^2$. The construction of the testing procedure follows the exact same idea of (\ref{eq:test-regression}), but with more abstract notation. For any $Z,Z^*\in\mathcal{Z}$, we use the $P_{Z\cup Z^*}\in\mathbb{R}^{N\times N}$ for the projection matrix onto the subspace spanned by the columns of $\X_Z$ and $\X_{Z^*}$. Define
$$\phi_Z=\indc{\|P_{Z\cup Z^*}(Y-\X_{Z^*}(B^*))\|^2>6(\epsilon(Z)^2+\epsilon(Z^*)^2)},$$
the the testing procedure is
\begin{equation}
\phi = \max_{Z\in\mathcal{Z}}\phi_Z. \label{eq:def-test-slm}
\end{equation}
The analysis of the testing error is also the same as that of (\ref{eq:test-regression}), and we state the result below.
\begin{lemma}\label{lem:slm-test}
Assume (\ref{eq:G-vdV-Z}). Under the setting of structured linear models, the testing procedure (\ref{eq:def-test-slm}) satisfies
\begin{eqnarray*}
P_{\X_{Z^*}(B^*)}\phi &\leq& \exp(-\epsilon(Z^*)^2 + 1), \\
\sup_{\theta\in\Theta_Z: \|\theta-\X_{Z^*}(B^*)\|^2\geq\epsilon^2}P_{\theta}(1-\phi) &\leq& \exp\left(-\frac{2}{3}\epsilon^2 + 5(\epsilon(Z)^2+\epsilon(Z^*)^2)\right),
\end{eqnarray*}
for any $\epsilon^2\geq\epsilon(Z^*)^2$ and any $Z\in\mathcal{Z}$.
\end{lemma}

Last but not least, we verify the condition (\ref{eq:EB-pm}) with $\rho=2$ so that $D_{\rho}(P_{\theta^*}\|P_{\theta})=\|\theta-\theta^*\|^2$. Recall that for any $Z\in\bar{\Z}_{\lambda}$, $\Gamma_Z$ is the probability measure of $\X_Z(B)$ with $B\sim f_{\ell(\Z_{\lambda}),\X_Z,\tau}$.
\begin{lemma}\label{lem:slm-pm}
For any $\tau>0$, there exists some constant $C_2>0$ only depending on $\tau$, such that
$$\frac{\Gamma_Z\left(\left\{\theta\in\Theta_Z: \|\theta-\X_{Z^*}(B^*)\|^2\leq \epsilon^2\right\}\right)}{\Gamma_{Z^*}\left(\left\{\theta\in\Theta_{Z^*}: \|\theta-\X_{Z^*}(B^*)\|^2\leq \epsilon(Z^*)^2\right\}\right)}\leq \frac{\delta(Z)}{\delta(Z^*)}\exp\left(\frac{1}{6}\epsilon^2 + C_2\left(\epsilon(Z^*)^2+\epsilon(Z)^2\right)\right),$$
for any $\epsilon^2\geq \epsilon(Z^*)^2$ and any $Z\in\mathcal{Z}$.
\end{lemma}

Lemma \ref{lem:slm-sieve}, Lemma \ref{lem:slm-test} and Lemma \ref{lem:slm-pm} together imply that the three conditions of Theorem \ref{thm:EB-con} are satisfied for some constants $C,C_1,C_2,C_3,C_4,C_5>0$. We can then immediately write down a theorem for the convergence rate of the empirical Bayes posterior by applying Theorem \ref{thm:EB-con} to the current setting.

\begin{thm}\label{thm:slm-EB}
Consider the empirical Bayes procedure defined with the elliptical Laplace distribution (\ref{eq:def-e-lap}) and the weight function (\ref{eq:fancy-weight}) satisfying $D\geq D_0$ for some constant $D_0>0$ only depending on $\tau$. Assume (\ref{eq:G-vdV-Z}). Then, the conditions (\ref{eq:EB-test})-(\ref{eq:EB-pm}) are satisfied under the setting of structured linear models. As a consequence, we have
$$P_{\X_{Z^*}(B^*)}\Pi_{\wh{\lambda}}\left(\|\X_Z(B)-\X_{Z^*}(B^*)\|^2 > M\epsilon(\Z_{\lambda^*})^2\Big|Y\right)\leq 6\exp(-C'\epsilon(\Z_{\lambda^*})^2),$$
with some constants $M,C'>0$ uniformly over all $B^*\in\mathbb{R}^{\ell(\Z_{\lambda^*})}$ and all $Z^*\in\bar{\Z}_{\lambda^*}$.
\end{thm}

Theorem \ref{thm:slm-EB} shows that the empirical Bayes procedure enjoys the same convergence rate as the hierarchical Bayes posterior in (\ref{eq:slm-hb-rate}). To close this section, we briefly discuss the implications of three examples. More examples in the framework of structured linear models are referred to \cite{gao2015general}.

\begin{example}[Sparse linear regression]
Consider a regression problem with fixed design $X\beta$, where $X\in\mathbb{R}^{n\times p}$ and $\beta\in\mathbb{R}^p$. A sparse regression coefficient vector can be written as $\beta^T=(\beta_S^T,0_{S^c}^T)$ for some $S\subset[p]$. This is a special case of the general structured linear model with $N=n$, $Z=S$, $\lambda=s$, $\Lambda=[p]$, $\mathcal{Z}_{s}=\{S\subset[p]:|S|=s\}$, $\ell(\Z_s)=s$, $B=\beta_S$ and the linear operator $\X_S:\beta_S\mapsto X_S\beta_S$. The empirical Bayes posterior distribution selects the hyperparameter $\wh{s}$, and achieves the rate $\ell(\Z_s)+\log|\Z_s|=s+\log{p\choose s}\asymp s\log\left(\frac{ep}{s}\right)$. The framework of structured linear models gives an alternative way to achieve the optimal rate in addition to the procedure induced by the spike-and-slab prior (Theorem \ref{eq:EB-reg}).
\end{example}

\begin{example}[Biclustering]
In a biclustering model, the observation is a matrix $Y\in\mathbb{R}^{n\times m}$. For any $i\in[n]$ and $j\in[m]$, $Y_{ij}\sim N(B_{z_1(i),z_2(j)},1)$ for some label vectors $z_1\in[k]^n$ and $z_2\in[l]^m$ and a matrix $B\in\mathbb{R}^{k\times l}$. In other words, there are $k$ row clusters and $l$ column clusters, and the mean of $Y$ admits a checkerboard structure. The biclustering model can be viewed as a special example of the general framework with $N=nm$, $Z=(z_1,z_2)$, $\lambda=(k,l)$, $\Lambda=[n]\times [m]$, $\Z_{k,l}=[k]^n\times [l]^m$, $\Z_{k,l}=kl$, and the linear operator $\X_{z_1,z_2}:B \mapsto (B_{z_1(i)z_2(j)})_{(i,j)\in[n]\times [m]}$. A careful reader may notice that the structured linear model framework works with $B$ and $\X_Z(B)$ that are vectors, but these two objects are matrices in the context of biclustering. This issue is only a matter of representation, and can be resolved by alternative notation using vectorization and Kronecker products. When specialized to the biclustering problem, the empirical Bayes procedure selects both the number of the row clusters and the number of column clusters, and achieves the rate $\ell(\Z_{k,l})+\log|\Z_{k,l}|=kl+k\log n+l\log m$, which is the minimax rate of biclustering \citep{gao2016optimal}.
\end{example}

\begin{example}[Multi-task learning with group sparsity]
In multitask learning, one observe a matrix $Y\in\mathbb{R}^{n\times m}$, whose mean is modeled by $XA$ with some design matrix $X\in\mathbb{R}^{n\times p}$ and regression coefficient matrix $A\in\mathbb{R}^{p\times m}$. We assume that the $m$ regression problems share the same sparsity pattern, which can be modeled by a group sparse structure. In other words, there is some $S\subset[p]$ such that the matrix $A$ has zero entries for all rows in $S^c$. We use the notation $A_{S*}$ for the submatrix of $A$ with rows in $S$. To put the problem into the general framework, let $Z=S$, $\lambda=s$, $\Lambda=[p]$, $\Z=\{S\subset[p]:|S|=s\}$, $\ell(\Z_s)=ms$ and $B=A_{S*}$. The linear operator is given by $\X_S:A_{S*}\mapsto X_SA_{S*}$. Similar to sparse linear regression, here, the empirical Bayes procedure also selects the sparsity, and achieves the rate $\ell(\Z_s)+\log|\Z_s|=ms+\log{p\choose s}\asymp s\left(m+ \log\left(\frac{ep}{s}\right)\right)$. This rate is known to be the minimax rate of the problem \citep{lounici2011oracle}.
\end{example}

\section{Additional Proofs}\label{sec:pf}

\subsection{Proofs of Theorem \ref{thm:den-exp-EB} and Theorem \ref{thm:den-mixture-EB}}

\begin{proof}[Proof of Theorem \ref{thm:den-exp-EB}]
We will apply Theorem \ref{thm:EB-ms} to prove this result. To put the infinite dimensional exponential families into the framework of Theorem \ref{thm:EB-ms}, we have $\Theta^{(k)}=\{\theta=(\theta_j)_{j=1}^{\infty}: \theta_j=0\text{ for all }j>k\}$ and $\theta^{(k)}=\theta$ for all $k$. Moreover, we have $\pi(k)\propto w(k)$ and thus $\pi(k)=\frac{\tau^k}{k!}e^{-\tau}$.
By Lemma B.10 of \cite{zhang2017convergence}, the two conditions (\ref{eq:EB-con-test}) and (\ref{eq:EB-con-sieve}) hold with $\epsilon_*^2=n\left(\frac{\log n}{n}\right)^{\frac{2\alpha}{2\alpha+1}}$ and loss $L(P_{\theta^*}^n,P_{\theta}^n)=nH^2(P_{\theta},P_{\theta^*})$. Thus, we only need to check the prior mass condition (\ref{eq:EB-pm-0}). Take $\rho=2$ and $k^*=\ceil{(n/\log n)^{\frac{1}{2\alpha+1}}}$. Then,
\begin{equation}
\Pi^{(k^*)}\left(\left\{\theta\in\Theta^{(k^*)}: D_{2}(P_{\theta^*}^n\|P_{\theta}^n)\leq C_3\epsilon_*^2\right\}\right)\geq \Pi^{(k^*)}(\wt{\Theta}), \label{eq:aos-exp}
\end{equation}
where
$$\wt{\Theta}=\left\{\theta=(\theta_j)_{j=1}^{\infty}: \theta_j\in[\theta_j^*-n^{-1/2},\theta_j^*+n^{-1/2}]\text{ for all }j\leq k\text{ and }\theta_j=0\text{ for all }j>k\right\}.$$
The inequality (\ref{eq:aos-exp}) holds since for any $\theta\in\wt{\Theta}$,
\begin{eqnarray}
\nonumber D_{2}(P_{\theta^*}^n\|P_{\theta}^n) &=& nD_{2}(P_{\theta^*}\|P_{\theta}) \\
\label{eq:lemma-B.12} &\leq& C_0n\exp(3\sqrt{2}\|\theta^*-\theta\|_1)\|\theta-\theta^*\|^2\\
\nonumber & = &C_0n\exp\left(3\sqrt{2}\left(\frac{k^*}{\sqrt{n}}+\sum_{j>k^*}|\theta_j^*|\right)\right)\left(\frac{k^*}{n}+\sum_{j>k^*}\theta_j^{*2}\right)\\
\nonumber &\leq&C_0n\exp\left(3\sqrt{2}\left(n^{\frac{1-2\alpha}{2+4\alpha}}+R\gamma_\alpha^{1/2}\right)\right)\left(\frac{k^*}{n}+(k^*)^{-2\alpha}R^2\right)\\
\nonumber &\leq& C_3\epsilon_*^2,
\end{eqnarray}
where $\gamma_{\alpha}=\sum_{j=1}^{\infty}j^{-2\alpha}$ is a constant for $\alpha>1/2$, and the inequality (\ref{eq:lemma-B.12}) is by Lemma B.12 of \cite{zhang2017convergence}. By Lemma B.8 of \cite{zhang2017convergence}, we have $\Pi^{(k^*)}(\wt{\Theta})\geq \exp(-C_2''\epsilon_*^2)$, and thus (\ref{eq:EB-pm-parameter}) holds. It is obvious that $\pi(k^*)=\frac{\tau^{k^*}}{k^*!}e^{-\tau}\geq \exp(-C_2'\epsilon_*^2)$, which implies (\ref{eq:EB-pm-model}). The condition (\ref{eq:EB-pm-0}) is a consequence of (\ref{eq:EB-pm-model}) and (\ref{eq:EB-pm-parameter}), and thus the proof is complete.
\end{proof}

\begin{proof}[Proof of Theorem \ref{thm:den-mixture-EB}]
Following the strategy of \cite{zhang2017convergence}, we introduce a surrogate density function $\wt{f}^*$ that is sufficiently close to $f^*$ and then apply Theorem \ref{thm:EB-ms} to $\wt{f}^*$. Define $\wt{f}^*(x)=\frac{f^*(x)\indc{x\in E}}{\int_E f^*(x)dx}$ with $E=\{x:f^*(x)\geq n^{-4}(\log n)^{4r}\}$ with the same $r$ defined in \cite{zhang2017convergence}. Specifically, $r$ is defined as $r = \frac{p}{\min\{p,\xi_3\}}+\max\{d_3+1,\frac{c_6}{\min\{p,\xi_3\}}\}$ with $p$, $\xi_3$ defined in (\ref{eq:kernel}) and (\ref{eq:true-condition3}) and $c_6$, $d_3$ defined in \cite{zhang2017convergence}. With the normal prior on $\mu$'s and Dirichlet distribution on $w$, one can verify that $c_6 = 2$ and $d_3=0$. Then $r = \frac{p}{\min\{p,\xi_3\}}+\max\{1,\frac{2}{\min\{p,\xi_3\}}\}$.

By the same argument used in the proof of Theorem 4.2 of \cite{zhang2017convergence}, we have
\begin{eqnarray}
\nonumber && P_{f^*}^n\int H^2(P_{\wh{k},\theta^{(\wh{k})}},P_{f^*})d\Pi^{(\wh{k})}(\theta^{(\wh{k})}|X_1,\cdots,X_n) \\
\label{eq:very-long} &\leq& 2P_{\wt{f}^*}^n\int H^2(P_{\wh{k},\theta^{(\wh{k})}},P_{\wt{f}^*})d\Pi^{(\wh{k})}(\theta^{(\wh{k})}|X_1,\cdots,X_n) + o\left(n^{-\frac{2\alpha}{2\alpha+1}}(\log n)^{\frac{2\alpha r}{\alpha+1}}\right).
\end{eqnarray}
We apply Theorem \ref{thm:EB-ms} to bound the first term of (\ref{eq:very-long}). Recall that for the location-scale matrix model,
\begin{eqnarray*}
\Theta^{(k)} &=& \Big\{\theta^{(k)} = (\mu,w,\sigma): \mu = (\mu_1,\cdots, \mu_k)\in\R^k, \\
\nonumber&& \qquad\qquad w = (w_1,\cdots, w_k)\in\Delta_k, \sigma\in\R_+\Big\},
\end{eqnarray*}
and $\pi(k)=\frac{\xi_0^k}{k!}e^{-\xi_0}$.
By Lemma B.15 of \cite{zhang2017convergence}, the two conditions (\ref{eq:EB-con-test}) and (\ref{eq:EB-con-sieve}) hold with $\epsilon_*^2=n^{\frac{1}{2\alpha+1}}(\log n)^{\frac{2\alpha r}{\alpha+1}}$ and loss $L(P_{f^*}^n,P_{k,\theta^{(k)}}^n)=nH^2(P_{f^*},P_{k,\theta^{(k)}})$. Moreover, Lemma B.16 of \cite{zhang2017convergence} shows that there exists some $k^*$ and some set
\begin{equation}
\wt{\Theta}\subset\left\{\theta^{(k^*)}\in\Theta^{(k^*)}: D_2(P_{\wt{f}^*}^n\|P_{k^*,\theta^{(k^*)}}^n)\leq C_3\epsilon_*^2\right\}, \label{eq:subset-good}
\end{equation}
such that $\pi(k^*)\Pi^{(k^*)}(\wt{\Theta})\geq \exp(-C_2\epsilon_*^2)$. The property (\ref{eq:subset-good}) of $\wt{\Theta}$ immediately implies the condition (\ref{eq:EB-pm-0}), and thus we have 
$$P_{\wt{f}^*}^n\int H^2(P_{\wh{k},\theta^{(\wh{k})}},P_{\wt{f}^*})d\Pi^{(\wh{k})}(\theta^{(\wh{k})}|X_1,\cdots,X_n)\lesssim n^{-\frac{2\alpha}{2\alpha+1}}(\log n)^{\frac{2\alpha r}{\alpha+1}},$$
according to Theorem \ref{thm:EB-ms}.
The proof is thus complete.
\end{proof}

\subsection{Proofs of Lemmas \ref{lem:spsl-sum}-\ref{lem:slm-pm}}

The proofs of Lemmas \ref{lem:spsl-sum}-\ref{lem:slm-pm} are given below. Note that the proof of Lemma \ref{lem:slm-sieve} is already stated in Section \ref{sec:slm}.

\begin{proof}[Proof of Lemma \ref{lem:spsl-sum}]
Choose $\lambda^*=\frac{\alpha+|S^*|-1}{p+\alpha+\beta-2}$, and then we have $\nu_{\lambda^*}(S^*)w(\lambda^*)=\gamma(S^*)$. We have
\begin{equation}
\sum_{S\subset[p]}\frac{\gamma(S)}{w(\lambda^*)\nu_{\lambda^*}(S^*)}\exp\left(2C_2|S|\log p\right) = \sum_{s=1}^p\sum_{S\subset[p]:|S|=s}\frac{\gamma(S)}{\gamma(S^*)}\exp(2C_2s\log  p). \label{eq:ensiferum}
\end{equation}
Since $\gamma(S)$ only depends on the cardinality of $S$, we use the notation $\gamma_{|S|}=\gamma(S)$.
For any $s\in[p]$, we have
\begin{eqnarray*}
\frac{\gamma_{s+1}}{\gamma_s} &=& \left(1+\frac{1}{\alpha+s-1}\right)^{\alpha+s-1}\left(1+\frac{1}{p-s+\beta-2}\right)^{-(p-s+\beta-2)}\frac{\alpha+s}{p-s+\beta-1}.
\end{eqnarray*}
As $(1+1/n)^n<e$ for any $n>0$, using the conditions of $\alpha$ and $\beta$, we have
$$p^{-C_1'}\leq \frac{\alpha}{e(p+\beta-1)}\leq \frac{\alpha+s}{e(p-s+\beta-1)}\leq\frac{\gamma_{s+1}}{\gamma_s}\leq \frac{(\alpha+s)e}{p-s+\beta-1}\leq \frac{e(\alpha+p)}{\beta-1}\leq p^{-C_2'},$$
for some constant $C_1'>C_2'>2C_2+1$. Then, we can further bound (\ref{eq:ensiferum}) by
\begin{eqnarray*}
\sum_{s=1}^p{p\choose s}\frac{\gamma_s}{\gamma_{s^*}}\exp(2C_2s\log p) &\leq& \exp(C_1's^*\log p)\sum_{s=1}^p\exp(-(C_1'-2C_2-1)s\log p) \\
&\leq& \exp(C_4s^*\log p),
\end{eqnarray*}
which is the desired result.
\end{proof}

\begin{proof}[Proofs of Lemma \ref{lem:spsl-seq-test} and Lemma \ref{lem:spsl-reg-test}]
We note that Lemma \ref{lem:spsl-seq-test} is a special case of Lemma \ref{lem:spsl-reg-test} with $X=I_p$ and $n=p$, and thus we only prove Lemma \ref{lem:spsl-reg-test}.
By Lemma 1 of \cite{laurent2000adaptive}, we have $\mathbb{P}\left(\chi_d^2 \geq d+2\sqrt{xd}+2x\right)\leq e^{-x}$
for any $x>0$.
This implies $\mathbb{P}(\chi_d^2>t)\leq \exp\left(\frac{2}{3}d-\frac{t}{3}\right)$ for any $t>2d$. The same bound also holds for $t\leq 2d$ since $\mathbb{P}(\chi_d^2>t)\leq 1\leq\exp\left(\frac{2}{3}d-\frac{t}{3}\right)$.
For the Type-1 error, we have
\begin{eqnarray*}
P_{\theta^*}\phi &\leq& \sum_{s=1}^p\sum_{S\subset[p]:|S|=s}P_{\theta^*}\left(\|P_{S\cup S^*}(Y-X\theta^*)\|^2 > 6\left(s\log p + s^*\log p\right)\right) \\
&\leq& \sum_{s=1}^p{p\choose s}\mathbb{P}\left(\chi_{s+s^*}^2>6\left(s\log p + s^*\log p\right)\right) \\
&\leq& \sum_{s=1}^p\exp\left(s\log p+\frac{2}{3}(s+s^*)-2(s\log p+s^*\log p)\right) \\
&\leq& \exp(-s^*\log p)\sum_{s=1}^p\exp\left(-\frac{1}{3}s\log p\right) \\
&\leq& \exp(-s^*\log p),
\end{eqnarray*}
where the last inequality assumes that $p\geq 8$. Now we analyze the Type-2 error. For any $\theta$ whose support is $S$, we write $|S|=s$. Then,
\begin{eqnarray*}
P_{\theta}(1-\phi) &\leq& P_{\theta}(1-\phi_S) \\
&\leq& P_{\theta}\left(\|X\theta-X\theta^*\|^2-\frac{1}{2}\|P_{S\cup S^*}(Y-X\theta)\|^2\leq 6\left(s\log p + s^*\log p\right)\right) \\
&\leq& \mathbb{P}\left(\chi_{s+s^*}^2>2\epsilon^2-12(s\log p +s^*\log p)\right) \\
&\leq& \exp\left(-\frac{2}{3}\epsilon^2 + 5(s\log p+s^*\log p)\right).
\end{eqnarray*}
The proof is complete.
\end{proof}

\begin{proof}[Proof of Lemma \ref{lem:spsl-seq-pm}]
Let us use the notation $s=|S|$ and $s^*=|S^*|$.
For any $\theta\in\Theta_S$ such that $\|\theta-\theta^*\|^2\leq\epsilon^2$ and any $\bar{\theta}\in\Theta_{S^*}$ such that $\|\bar{\theta}-\theta^*\|^2\leq s^*\log p$, we have
\begin{eqnarray*}
-\|\theta\|_1 + \|\bar{\theta}\|_1 &\leq& \|\theta-\bar{\theta}\|_1 \\
&\leq& \sqrt{s+s^*}\|\theta-\bar{\theta}\| \\
&\leq& \sqrt{s+s^*}\left(\epsilon + \sqrt{s^*\log p}\right) \\
&\leq& \xi\epsilon^2 + \left(\frac{1}{4\xi}+\frac{1}{2}\right)s+\left(1+\frac{1}{4\xi}\right)s^*\log p\\
&\leq&\xi\epsilon^2+M_1(s+s^*\log p),
\end{eqnarray*}
where $M_1 = 1+\frac{1}{4\xi}$ with $\xi>0$ to be determined later. 
Therefore,
\begin{eqnarray}
\nonumber && \frac{\Gamma_S\left(\left\{\theta\in\Theta_S:\|\theta-\theta^*\|^2\leq\epsilon^2\right\}\right)}{\Gamma_{S^*}\left(\theta\in\Theta_{S^*}:\|\theta-\theta^*\|^2\leq s^*\log p\right)} \\
\nonumber &=& \left(\frac{\tau}{2}\right)^{s-s^*}\frac{\int_{\|\theta-\theta^*\|^2\leq\epsilon^2}\exp(-\tau\|\theta\|_1)d\theta_S}{\int_{\|\theta-\theta^*\|^2\leq s^*\log p} \exp(-\tau\|\theta\|_1)d\theta_{S^*}} \\
\nonumber &\leq& \left(\frac{\tau}{2}\right)^{s-s^*}\exp\left(\tau\left(\xi\epsilon^2 +M_1(s+s^*\log p)\right)\right)\frac{\text{Vol}\left(\{\theta\in\Theta_S:\|\theta\|^2\leq \epsilon^2\}\right)}{\text{Vol}\left(\left\{\theta\in\Theta_{S^*}:\|\theta\|^2\leq s^*\log p\right\}\right)} \\
\nonumber &\leq& \left(\frac{1}{2}\right)^{s-s^*}\exp\left(\tau\left(\xi\epsilon^2 +M_1(s+s^*\log p)\right)\right)\frac{(2e\pi)^{s/2}\exp\left(s\log\frac{\epsilon}{s}\right)}{(2\pi e)^{s^*/2}\exp\left(s^*\log\frac{\sqrt{s^*\log p}}{s^*}\right)} \\
\label{eq:careful} &\leq& \left(\frac{1}{2}\right)^{s-s^*}\exp\left(\tau\left(\xi\epsilon^2 + M_1(s+s^*\log p)\right)\right)\frac{(2e\pi)^{s/2}\exp\left(\epsilon\right)}{(2\pi e)^{s^*/2}\exp\left(-\frac{1}{2}s^*\log s^*\right)}  \\
&\leq&\left(\frac{1}{2}\right)^{s-s^*}\exp\left(\tau\left(\xi\epsilon^2 + M_1(s+s^*\log p)\right)\right)\exp\left(\frac{s}{2}\log(2\pi e)+\xi\epsilon^2+\frac{1}{4\xi}+\frac{1}{2}s^*\log p\right)\nonumber\\
&\leq& \exp\left((\tau+1)\xi\epsilon^2+\left(M_1+\frac{1}{2}\log(2\pi e)\right)s+\left(M_1+\frac{1}{4\xi}+\frac{1}{2}\right)s^*\log p\right).\nonumber
\end{eqnarray}
The inequality (\ref{eq:careful}) uses the fact that the $\log x< x$ for all $x>0$.
Choosing $\xi = \frac{1}{6(\tau+1)}$ and $C_2 = M_1+\frac{1}{4\xi}+\frac{1}{2}\log(2\pi e)$ that only depend on $\tau$, we have
$$\frac{\Gamma_S\left(\left\{\theta\in\Theta_S:\|\theta-\theta^*\|^2\leq\epsilon^2\right\}\right)}{\Gamma_{S^*}\left(\theta\in\Theta_{S^*}:\|\theta-\theta^*\|^2\leq s^*\log p\right)} \leq\exp\left(\frac{1}{6}\epsilon^2+C_2(s+s^*\log p)\right).$$
The proof is complete.
\end{proof}

\begin{proof}[Proof of Lemma \ref{lem:spsl-reg-pm}]
We will use the notation $s=|S|$ and $s^*=|S^*|$ in the proof.
We first analyze the numerator. We have
\begin{eqnarray}
\nonumber && \Gamma_S\left(\left\{\theta\in\Theta_S: \|X(\theta-\theta^*)\|^2\leq\epsilon^2\right\}\right) \\
\nonumber &\leq&  \Gamma_S\left(\left\{\theta\in\Theta_S: \|X(\theta-\theta^*)\|^2\leq\epsilon^2, \|(\theta-\theta^*)_{S^{*c}}\|_1\leq 3\|(\theta-\theta^*)_{S^*}\|_1\right\}\right) \\
\nonumber && +  \Gamma_S\left(\left\{\theta\in\Theta_S: \|X(\theta-\theta^*)\|^2\leq\epsilon^2, \|(\theta-\theta^*)_{S^{*c}}\|_1> 3\|(\theta-\theta^*)_{S^*}\|_1\right\}\right) \\
\label{eq:use-comp} &\leq& \Gamma_S\left(\left\{\theta\in\Theta_S: \|\theta-\theta^*\|_1\leq \frac{4|S^*|^{1/2}\epsilon}{\|X\|\kappa(S^*)}\right\}\right) \\
\label{eq:drop-event} && + \Gamma_S\left(\left\{\theta\in\Theta_S: \|(\theta-\theta^*)_{S^{*c}}\|_1> 3\|(\theta-\theta^*)_{S^*}\|_1\right\}\right).
\end{eqnarray}
The inequality (\ref{eq:use-comp}) is by
$$\|\theta-\theta^*\|_1\leq 4\|(\theta-\theta^*)_{S^*}\|_1\leq \frac{4|S^*|^{1/2}\|X(\theta-\theta^*)\|}{\|X\|\kappa(S^*)}\leq \frac{4|S^*|^{1/2}\epsilon}{\|X\|\kappa(S^*)},$$
where we have used the definition of the compatibility constant $\kappa(S^*)$. We shall bound the two terms (\ref{eq:use-comp}) and (\ref{eq:drop-event}) separately. For (\ref{eq:use-comp}), we have
\begin{eqnarray*}
&& e^{\bar{\tau}\|X\|\|\theta^*\|_1}\Gamma_S\left(\left\{\theta\in\Theta_S: \|\theta-\theta^*\|_1\leq \frac{4|S^*|^{1/2}\epsilon}{\|X\|\kappa(S^*)}\right\}\right) \\
&=& \left(\frac{\bar{\tau}\|X\|}{2}\right)^s\int_{\|\theta-\theta^*\|_1\leq \frac{4|S^*|^{1/2}\epsilon}{\|X\|\kappa(S^*)}}\exp\left(\bar{\tau}\|X\|\|\theta^*\|_1-\bar{\tau}\|X\|\|\theta\|_1\right)d\theta_S \\
&\leq& \left(\frac{\bar{\tau}\|X\|}{2}\right)^s\int_{\|\theta-\theta^*\|_1\leq \frac{4|S^*|^{1/2}\epsilon}{\|X\|\kappa(S^*)}}\exp\left(\bar{\tau}\|X\|\|\theta-\theta^*\|_1\right)d\theta_S \\
&\leq& \left(\frac{\bar{\tau}\|X\|}{2}\right)^s\exp\left(\frac{8\bar{\tau}\sqrt{s^*}\epsilon}{\kappa(S^*)}\right)\int_{\|\theta-\theta^*\|_1\leq \frac{4|S^*|^{1/2}\epsilon}{\|X\|\kappa(S^*)}}\exp\left(-\bar{\tau}\|X\|\|\theta-\theta^*\|_1\right)d\theta_S \\
&\leq& \exp\left(\frac{8\bar{\tau}\sqrt{s^*}\epsilon}{\kappa(S^*)}\right)\left(\frac{\bar{\tau}\|X\|}{2}\right)^s\int\exp\left(-\bar{\tau}\|X\|\|\theta\|_1\right)d\theta_S \\
&=& \exp\left(\frac{8\bar{\tau}\sqrt{s^*}\epsilon}{\kappa(S^*)}\right).
\end{eqnarray*}
We also bound (\ref{eq:drop-event}) by
\begin{eqnarray}
\nonumber && e^{\bar{\tau}\|X\|\|\theta^*\|_1}\Gamma_S\left(\left\{\theta\in\Theta_S: \|(\theta-\theta^*)_{S^{*c}}\|_1> 3\|(\theta-\theta^*)_{S^*}\|_1\right\}\right) \\
\nonumber &=& \left(\frac{\bar{\tau}\|X\|}{2}\right)^s\int_{\|(\theta-\theta^*)_{S^{*c}}\|_1> 3\|(\theta-\theta^*)_{S^*}\|_1}\exp\left(\bar{\tau}\|X\|\|\theta^*\|_1-\bar{\tau}\|X\|\|\theta\|_1\right)d\theta_S \\
\label{eq:cone-use} &\leq& \left(\frac{\bar{\tau}\|X\|}{2}\right)^s\int\exp\left(-\frac{\bar{\tau}\|X\|}{2}\|\theta-\theta^*\|_1\right)d\theta_S \\
\nonumber &\leq& \left(\frac{\bar{\tau}\|X\|}{2}\right)^s\int\exp\left(-\frac{\bar{\tau}\|X\|}{2}\|\theta\|_1\right)d\theta_S \\
\nonumber &=& 2^s,
\end{eqnarray}
where the inequality is by (\ref{eq:cone-use})
\begin{eqnarray*}
\|\theta^*\|_1 - \|\theta\|_1 &=& \|\theta^*_{S^*}\|_1 - \|\theta_{S^*}^*+(\theta-\theta^*)_{S^*}\|_1 - \|(\theta-\theta^*)_{S^{*c}}\|_1 \\
&\leq& \|(\theta-\theta^*)_{S^*}\|_1- \|(\theta-\theta^*)_{S^{*c}}\|_1 \\
&=& \frac{3}{2}\|(\theta-\theta^*)_{S^*}\|_1 - \frac{1}{2}\|(\theta-\theta^*)_{S^{*c}}\|_1 - \frac{1}{2}\|\theta-\theta^*\|_1 \\
&\leq& - \frac{1}{2}\|\theta-\theta^*\|_1.
\end{eqnarray*}
Combine the two bounds above, we have
\begin{equation}
e^{\bar{\tau}\|X\|\|\theta^*\|_1}\Gamma_S\left(\left\{\theta\in\Theta_S: \|X(\theta-\theta^*)\|^2\leq\epsilon^2\right\}\right) \leq 2^s + \exp\left(\frac{8\bar{\tau}\sqrt{s^*}\epsilon}{\kappa(S^*)}\right).\label{eq:num-up}
\end{equation}
Next, we analyze the denominator. For any $\theta\in\Theta_{S^*}$, we have
$$\|X(\theta-\theta^*)\|\leq \|X\|\|\theta-\theta^*\|_1 \leq s^*\|X\|\|\theta-\theta^*\|_{\infty}.$$
Then,
\begin{eqnarray*}
&& e^{\bar{\tau}\|X\|\|\theta^*\|_1}\Gamma_{S^*}\left(\left\{\theta\in\Theta_{S^*}:\|X(\theta-\theta^*)\|^2\leq |S^*|\log p\right\}\right) \\
&=& \left(\frac{\bar{\tau}\|X\|}{2}\right)^{s^*}\int_{\|X(\theta-\theta^*)\|^2\leq s^*\log p}\exp\left(\bar{\tau}\|X\|\|\theta^*\|_1-\bar{\tau}\|X\|\|\theta\|_1\right)d\theta_{S^*} \\
&\geq& \left(\frac{\bar{\tau}\|X\|}{2}\right)^{s^*}\int_{s^*\|X\|^2\|\theta-\theta^*\|_{\infty}^2\leq \log p}\exp(-\bar{\tau}\|X\|\|\theta-\theta^*\|_1)d\theta_{S^*} \\
&=& \left(\frac{\bar{\tau}\|X\|}{2}\int_{-\sqrt{\frac{\log p}{s^*\|X\|^2}}}^{\sqrt{\frac{\log p}{s^*\|X\|^2}}}\exp\left(-\bar{\tau}\|X\||t|\right)dt\right)^{s^*} \\
&=& \left(\frac{1}{2}\int_{-\bar{\tau}\sqrt{\frac{\log p}{s^*}}}^{\bar{\tau}\sqrt{\frac{\log p}{s^*}}}e^{-|t|}dt\right)^{s^*}.
\end{eqnarray*}
For $\bar{\tau}=p^{-\zeta}$, we have
$$\frac{1}{2}\int_{-\bar{\tau}\sqrt{\frac{\log p}{s^*}}}^{\bar{\tau}\sqrt{\frac{\log p}{s^*}}}e^{-|t|}dt\geq \frac{1}{2}\int_{-\bar{\tau}\sqrt{\frac{\log p}{p}}}^{\bar{\tau}\sqrt{\frac{\log p}{p}}}e^{-|t|}dt \geq \bar{\tau}\sqrt{\frac{\log p}{p}} e^{-\bar{\tau}\sqrt{\frac{\log p}{p}}}\geq \exp\left(-\frac{C_2}{2}\log p\right),$$
for some constant $C_2$ depending on $\zeta$. Therefore,
\begin{equation}
e^{\bar{\tau}\|X\|\|\theta^*\|_1}\Gamma_{S^*}\left(\left\{\theta\in\Theta_{S^*}:\|X(\theta-\theta^*)\|^2\leq |S^*|\log p\right\}\right) \geq \exp\left(-\frac{C_2}{2}s^*\log p\right).\label{eq:den-low}
\end{equation}
Combine the two bounds (\ref{eq:num-up}) and (\ref{eq:den-low}), and we have
\begin{eqnarray*}
&& \frac{\Gamma_S\left(\left\{\theta\in\Theta_S: \|X(\theta-\theta^*)\|^2\leq\epsilon^2\right\}\right)}{\Gamma_{S^*}\left(\left\{\theta\in\Theta_{S^*}:\|X(\theta-\theta^*)\|^2\leq |S^*|\log p\right\}\right)} \\
&\leq& \frac{2^s + \exp\left(\frac{8\bar{\tau}\sqrt{s^*}\epsilon}{\kappa(S^*)}\right)}{\exp\left(-\frac{C_2}{2}s^*\log p\right)} \\
&\leq& \frac{2^s + \exp\left(\bar{\tau}\epsilon^2 + \frac{16\bar{\tau}s^*}{\kappa(S^*)}\right)}{\exp\left(-\frac{C_2}{2}s^*\log p\right)} \\
&\leq& \exp\left(\bar{\tau}\epsilon^2 + C_2(s\log p+ s^*\log p)\right),
\end{eqnarray*}
where the last inequality uses the condition that $\bar{\tau}/\kappa(S^*)\leq 1$. The proof is complete.
\end{proof}

\begin{proof}[Proof of Lemma \ref{lem:slm-test}]
The assumption (\ref{eq:G-vdV-Z}) implies that
\begin{equation}
\sum_{\lambda\in\Lambda}\exp\left(-\epsilon(\Z_{\lambda})^2\right) \leq \sum_{t=1}^{\infty}te^{-(t-1)} \leq \int_0^{\infty}te^{-(t-1)}dt = e. \label{eq:sum-to-int}
\end{equation}
By the same argument used in the proof of Lemma \ref{lem:spsl-reg-test}, we have $\mathbb{P}(\chi_d^2>t)\leq \exp\left(\frac{2}{3}d-\frac{t}{3}\right)$ for any $t>0$.
For the Type-1 error, we have
\begin{eqnarray*}
P_{\X_{Z^*}(B^*)}\phi &\leq& \sum_{\lambda\in\Lambda}\sum_{\lambda\in\bar{\Z}_{\lambda}}P_{\X_{Z^*}(B^*)}\left(\|P_{Z\cup Z^*}(Y-\X_{Z^*}(B^*))\|^2>6\left(\epsilon(\Z_{\lambda})^2+\epsilon(\Z_{\lambda^*})^2\right)\right) \\
&\leq& \sum_{\lambda\in\Lambda}\sum_{\lambda\in\bar{\Z}_{\lambda}}\mathbb{P}\left(\chi^2_{\ell(\Z_{\lambda})+\ell(\Z_{\lambda^*})}>6\left(\epsilon(\Z_{\lambda})^2+\epsilon(\Z_{\lambda^*})^2\right)\right) \\
&\leq& \sum_{\lambda\in\Lambda}\exp\left(\log|\Z_{\lambda}|+\frac{2}{3}\left(\ell(\Z_{\lambda})+\ell(\Z_{\lambda^*})\right)-2\left(\epsilon(\Z_{\lambda})^2+\epsilon(\Z_{\lambda^*})^2\right)\right) \\
&\leq& \exp\left(-\epsilon(\Z_{\lambda^*})^2\right)\sum_{\lambda\in\Lambda}\exp\left(-\epsilon(\Z_{\lambda})^2\right) \\
&\leq& \exp\left(-\epsilon(\Z_{\lambda^*})^2+1\right).
\end{eqnarray*}
Now we analyze the Type-2 error. For any $\theta\in\Theta_Z$ with $Z\in\bar{\Z}_{\lambda}$, there exists some $B\in\mathbb{R}^{\ell(\Z_{\lambda})}$ such that $\theta=\X_Z(B)$. Thus,
\begin{eqnarray*}
P_{\theta}(1-\phi) &\leq& P_{\X_Z(B)}(1-\phi_Z) \\
&\leq& P_{\X_Z(B)}\left(\|\X_Z(B)-\X_{Z^*}(B^*)\|^2-\frac{1}{2}\|P_{Z\cup Z^*}(Y-\X_{Z}(B))\|^2\leq 6\left(\epsilon(\Z_{\lambda})^2+\epsilon(\Z_{\lambda^*})^2\right)\right) \\
&\leq& \mathbb{P}\left(\chi^2_{\ell(\Z_{\lambda})+\ell(\Z_{\lambda^*})}>2\epsilon^2 - 12\left(\epsilon(\Z_{\lambda})^2+\epsilon(\Z_{\lambda^*})^2\right)\right) \\
&\leq& \exp\left(-\frac{2}{3}\epsilon^2 + 5\left(\epsilon(\Z_{\lambda})^2+\epsilon(\Z_{\lambda^*})^2\right)\right).
\end{eqnarray*}
The proof is complete.
\end{proof}

\begin{proof}[Proof of Lemma \ref{lem:slm-pm}]
We can write
\begin{eqnarray*}
&& \frac{\Gamma_Z\left(\left\{\theta\in\Theta_Z: \|\theta-\X_{Z^*}(B^*)\|^2\leq \epsilon^2\right\}\right)}{\Gamma_{Z^*}\left(\left\{\theta\in\Theta_{Z^*}: \|\theta-\X_{Z^*}(B^*)\|^2\leq \epsilon(Z^*)^2\right\}\right)} \\
&=& \frac{e^{\tau\|\X_{Z^*}(B^*)\|}\Gamma_Z\left(\left\{\theta\in\Theta_Z: \|\theta-\X_{Z^*}(B^*)\|^2\leq \epsilon^2\right\}\right)}{e^{\tau\|\X_{Z^*}(B^*)\|}\Gamma_{Z^*}\left(\left\{\theta\in\Theta_{Z^*}: \|\theta-\X_{Z^*}(B^*)\|^2\leq \epsilon(\Z_{\lambda^*})^2\right\}\right)},
\end{eqnarray*}
and we will analyze the numerator and the denominator separately. To facilitate the analysis for the numerator, we introduce the object
$$\bar{B}_Z=\argmin_{B\in\mathbb{R}^{\ell(\Z_{\lambda})}}\|\X_Z(B)-\X_{Z^*}(B^*)\|^2.$$
The property of least-squares implies the following Pythagorean identity,
$$\|\X_Z(B)-\X_{Z^*}(B^*)\|^2 = \|\X_Z(B)-\X_Z(\bar{B}_Z)\|^2 + \|\X_Z(\bar{B}_Z)-\X_{Z^*}(B^*)\|^2.$$
This implies
\begin{equation}
\|\X_Z(B)-\X_{Z^*}(B^*)\|^2 \geq \|\X_Z(B)-\X_Z(\bar{B}_Z)\|^2. \label{eq:my-old-stuff}
\end{equation}
Now we bound the numerator by
\begin{eqnarray}
\nonumber && e^{\tau\|\X_{Z^*}(B^*)\|}\Gamma_Z\left(\left\{\theta\in\Theta_Z: \|\theta-\X_{Z^*}(B^*)\|^2\leq \epsilon^2\right\}\right) \\
\nonumber &=& \frac{\sqrt{\det(\X_Z^T\X_Z)}}{2}\left(\frac{\tau}{\sqrt{\pi}}\right)^{\ell(\Z_{\lambda})}\frac{\Gamma(\ell(\Z_{\lambda})/2)}{\Gamma(\ell(\Z_{\lambda}))}\\
\nonumber && \times \int_{B\in\mathbb{R}^{\ell(\Z_{\lambda})}:\|\X_Z(B)-\X_{Z^*}(B^*)\|^2\leq\epsilon^2}e^{\tau\|\X_{Z^*}(B^*)\|-\tau\|\X_Z(B)\|}dB \\
\nonumber &\leq& e^{\tau\epsilon}\frac{\sqrt{\det(\X_Z^T\X_Z)}}{2}\left(\frac{\tau}{\sqrt{\pi}}\right)^{\ell(\Z_{\lambda})}\frac{\Gamma(\ell(\Z_{\lambda})/2)}{\Gamma(\ell(\Z_{\lambda}))}\int_{B\in\mathbb{R}^{\ell(\Z_{\lambda})}:\|\X_Z(B)-\X_{Z^*}(B^*)\|^2\leq\epsilon^2}dB \\
\label{eq:use-Py} &\leq& e^{\tau\epsilon}\frac{\sqrt{\det(\X_Z^T\X_Z)}}{2}\left(\frac{\tau}{\sqrt{\pi}}\right)^{\ell(\Z_{\lambda})}\frac{\Gamma(\ell(\Z_{\lambda})/2)}{\Gamma(\ell(\Z_{\lambda}))}\int_{B\in\mathbb{R}^{\ell(\Z_{\lambda})}:\|\X_Z(B)-\X_{Z}(\bar{B}_Z)\|^2\leq\epsilon^2}dB \\
\label{eq:change-var}&=& \frac{1}{2}e^{\tau\epsilon}\left(\frac{\tau}{\sqrt{\pi}}\right)^{\ell(\Z_{\lambda})}\frac{\Gamma(\ell(\Z_{\lambda})/2)}{\Gamma(\ell(\Z_{\lambda}))}\text{Vol}\left(\left\{B\in\mathbb{R}^{\ell(\Z_{\lambda})}:\|B\|\leq \epsilon\right\}\right) \\
\nonumber &\leq& \frac{\Gamma(\ell(\Z_{\lambda})/2)}{\Gamma(\ell(\Z_{\lambda}))}e^{\tau\epsilon}(2e)^{\ell(\Z_{\lambda})/2}\exp\left(\frac{1}{2}\ell(\Z_{\lambda})\log\frac{2\pi e\epsilon^2}{\ell(\Z_{\tau})}\right) \\
\nonumber &\leq& \frac{\Gamma(\ell(\Z_{\lambda})/2)}{\Gamma(\ell(\Z_{\lambda}))}e^{\tau\epsilon}(2e)^{\ell(\Z_{\lambda})/2}\exp\left(\ell(\Z_{\lambda})\log\frac{\sqrt{2\pi e}\epsilon}{\sqrt{\ell(\Z_{\tau})}}\right) \\
\label{eq:last-in} &\leq& \frac{\Gamma(\ell(\Z_{\lambda})/2)}{\Gamma(\ell(\Z_{\lambda}))}(2e)^{\ell(\Z_{\lambda})/2}\exp\left(\tau\epsilon + \sqrt{\ell(\Z_{\lambda})}\sqrt{2\pi e}\epsilon\right)\\
\nonumber &\leq& \frac{\Gamma(\ell(\Z_{\lambda})/2)}{\Gamma(\ell(\Z_{\lambda}))}e^{\tau\epsilon}(2e)^{\ell(\Z_{\lambda})/2}\exp\left(\ell(\Z_{\lambda})\log\frac{\sqrt{2\pi e}\epsilon}{\sqrt{\ell(\Z_{\tau})}}\right) \\
\nonumber &\leq& \frac{\Gamma(\ell(\Z_{\lambda})/2)}{\Gamma(\ell(\Z_{\lambda}))}(2e)^{\ell(\Z_{\lambda})/2}\exp\left(\frac{1}{12}\epsilon^2+3\tau^2+\frac{1}{12}\epsilon^2+6\pi e\ell(\Z_{\lambda})\right),
\end{eqnarray}
where (\ref{eq:use-Py}) is derived from (\ref{eq:my-old-stuff}), (\ref{eq:change-var}) is a standard change-of-variable argument, and the inequality (\ref{eq:last-in}) uses the fact that the function $\log x< x$ for any $x>0$. Then, we have
\begin{equation}
e^{\tau\|\X_{Z^*}(B^*)\|}\Gamma_Z\left(\left\{\theta\in\Theta_Z: \|\theta-\X_{Z^*}(B^*)\|^2\leq \epsilon^2\right\}\right) \leq \frac{\Gamma(\ell(\Z_{\lambda})/2)}{\Gamma(\ell(\Z_{\lambda}))}\exp\left(C_2\epsilon(\Z_{\lambda})^2 + \frac{1}{6}\epsilon^2\right), \label{eq:num-upp}
\end{equation}
for $C_2 > 3\tau^2+6\pi e$.
Similarly, for the denominator, we have
\begin{eqnarray*}
&& e^{\tau\|\X_{Z^*}(B^*)\|}\Gamma_{Z^*}\left(\left\{\theta\in\Theta_{Z^*}: \|\theta-\X_{Z^*}(B^*)\|^2\leq \epsilon(\Z_{\lambda^*})^2\right\}\right) \\
&\geq& \frac{\sqrt{\det(\X_{Z^*}^T\X_{Z^*})}}{2}\left(\frac{\tau}{\sqrt{\pi}}\right)^{\ell(\Z_{\lambda^*})}\frac{\Gamma(\ell(\Z_{\lambda^*})/2)}{\Gamma(\ell(\Z_{\lambda^*}))}\\
\nonumber && \times \int_{B\in\mathbb{R}^{\ell(\Z_{\lambda^*})}:\|\X_{Z^*}(B)-\X_{Z^*}(B^*)\|^2\leq \epsilon(\Z_{\lambda^*})^2}e^{\tau\|\X_{Z^*}(B^*)\|-\tau\|\X_{Z^*}(B)\|}dB \\
&\geq& \frac{1}{2}e^{\tau\epsilon(\Z_{\lambda^*})}\left(\frac{\tau}{\sqrt{\pi}}\right)^{\ell(\Z_{\lambda^*})}\frac{\Gamma(\ell(\Z_{\lambda^*})/2)}{\Gamma(\ell(\Z_{\lambda^*}))}\text{Vol}\left(\left\{B\in\mathbb{R}^{\ell(\Z_{\lambda^*})}:\|B\|^2\leq \epsilon(\Z_{\lambda^*})^2\right\}\right) \\
&\geq& \frac{\Gamma(\ell(\Z_{\lambda^*})/2)}{\Gamma(\ell(\Z_{\lambda^*}))}\frac{(2e\tau^2)^{\ell(\Z_{\lambda^*})/2}}{2\sqrt{\pi\ell(\Z_{\lambda^*})}}e^{\tau\epsilon(\Z_{\lambda^*})}\left(\frac{\epsilon(\Z_{\lambda^*})}{\sqrt{\ell(\Z_{\lambda^*})}}\right)^{\ell(\Z_{\lambda^*})} \\
&\geq& \frac{\Gamma(\ell(\Z_{\lambda^*})/2)}{\Gamma(\ell(\Z_{\lambda^*}))}\frac{(2e\tau^2)^{\ell(\Z_{\lambda^*})/2}}{2\sqrt{\pi\ell(\Z_{\lambda^*})}}e^{\tau\epsilon(\Z_{\lambda^*})^2},
\end{eqnarray*}
where the last inequality uses the fact that $\frac{\epsilon(\Z_{\lambda^*})^2}{\ell(\Z_{\lambda^*})}\geq 1$. Therefore,
\begin{equation}
e^{\tau\|\X_{Z^*}(B^*)\|}\Gamma_{Z^*}\left(\left\{\theta\in\Theta_{Z^*}: \|\theta-\X_{Z^*}(B^*)\|^2\leq \epsilon(\Z_{\lambda^*})^2\right\}\right)\geq \frac{\Gamma(\ell(\Z_{\lambda^*})/2)}{\Gamma(\ell(\Z_{\lambda^*}))}\exp(-C_2\epsilon(\Z_{\lambda^*})^2), \label{eq:den-loww}
\end{equation}
for $C_2\geq 2+\tau+\frac{1}{2}\left|\log(2e\tau^2)\right|$. Combine the two bounds (\ref{eq:num-upp}) and (\ref{eq:den-loww}), and we obtain
\begin{eqnarray*}
&& \frac{\Gamma_Z\left(\left\{\theta\in\Theta_Z: \|\theta-\X_{Z^*}(B^*)\|^2\leq \epsilon^2\right\}\right)}{\Gamma_{Z^*}\left(\left\{\theta\in\Theta_{Z^*}: \|\theta-\X_{Z^*}(B^*)\|^2\leq \epsilon(Z^*)^2\right\}\right)} \\
&\leq& \frac{\frac{\Gamma(\ell(\Z_{\lambda})/2)}{\Gamma(\ell(\Z_{\lambda}))}}{\frac{\Gamma(\ell(\Z_{\lambda^*})/2)}{\Gamma(\ell(\Z_{\lambda^*}))}}\exp\left(\frac{1}{6}\epsilon^2 + C_2\left(\epsilon(\Z_{\lambda^*})^2+\epsilon(\Z_{\lambda})^2\right)\right),
\end{eqnarray*}
for some $C_2>0$.
\end{proof}

\begin{small}
\bibliographystyle{plainnat}
\bibliography{reference}
\end{small}





\end{document}